\documentclass[12pt,sort&compress]{article}
\usepackage{amsmath,amsfonts,amssymb,amsthm}
\usepackage{graphicx,color,subfigure}
\usepackage[usenames,dvipsnames]{xcolor}
\usepackage{mathdots}
\usepackage{longtable}
\usepackage{array}
\usepackage{soul}
\usepackage{comment}

\newcommand{\hh}{h}

\newcommand{\dDG}{d} 
\newcommand{\ddg}{\dDG} 
\newcommand{\df}{k} 

\newcommand{\nc}{r} 
\newcommand{\drepro}{\nc} 
\newcommand{\idxfst}{0}
\newcommand{\idxlst}{{j_\nc}}

\newcommand{\nkft}{n} 
\newcommand{\lst}{\nkft} %
\newcommand{\sft}{j}

\newcommand{\nkdg}{m} 
\newcommand{\lstDG}{\nkdg}
\newcommand{\ffn}{f} 
\newcommand{\flt}[1]{\ffn_{#1}}
\newcommand{\cf}[2]{\ffn_{#1;#2}}
\newcommand{\cfno}[1]{\ffn_{#1}}

\newcommand{\lft}{a} 
\newcommand{\rgt}{b} 
\newcommand{\itmsym}{$\triangleright$} 
\newcommand{\idxft }{\scalebox{.7}{\ensuremath{\mathcal{J}}}} 

\newcommand{\ft}{filter}
\newcommand{\filter}{kernel}
\newcommand{\orderft}{reproduction degree}
\newcommand{\mult}{\nu}

\newcommand{\refx}{\lambda}
\newcommand{\ipmat}{G} 
\newcommand{\ipmatl}[1]{\ipmat_{\refx_{#1}}} 

\newcommand{\kft}{t}
\newcommand{\xx}{x}
\newcommand{\bkft}{\mathbf{\kft}}
\newcommand{\knotsftv}{\boldsymbol \kft}   

\newcommand{\tr}{\mathtt{t}}

\newcommand{\oldMethod}{numerical approach}
\newcommand{\newMethod}{symbolic approach}

\definecolor{dblue}{rgb}{0,.3,.6}
\definecolor{dgreen}{rgb}{0,.4,0}

\newcommand{\finaltime}{\widetilde{\tau}}
\newcommand{\alphab}{\boldsymbol\alpha}
\newcommand{\ix}{\omega}
\newcommand{\midx}{{\boldsymbol\omega}}
\newcommand{\tm}{\tau} 
\newcommand{\tmend}{\finaltime} 
\newcommand{\tb}{\mathbf{t}}
\newcommand{\ab}{\mathbf{a}}
\newcommand{\bb}{\mathbf{b}}

\newcommand{\lb}{\boldsymbol{\ell}}
\newcommand{\sk}{f} 

\newcommand{\drm}{\mathrm{d}} 

\newcommand{\Bsp}[2]{B(#1\,|\,#2)}
\newcommand{\Nbfi}[1]{\phi_{#1}}
\newcommand{\Nbf}[3]{\Nbfi{#1}(#2\,;\,#3)}
\newcommand{\Nsp}[2]{\Nbf{}{#1}{#2}}
\newcommand{\ub}{\mathbf{u}}
\newcommand{\Ical}{\mathcal{I}}

\newcommand{\diag}{\operatorname{diag}}
\newcommand{\adia}{A}
\newcommand{\kernel}{kernel}
\newcommand{\vb}{\mathbf{v}}

\newcommand{\figref}[1]{Fig.~\ref{#1}}
\newcommand{\secref}[1]{Section~\ref{#1}}
\newcommand{\thmref}[1]{Theorem~\ref{#1}}
\newcommand{\propref}[1]{Proposition~\ref{#1}}
\newcommand{\lemref}[1]{Lemma~\ref{#1}}
\newcommand{\equaref}[1]{Eq.~\eqref{#1}}
\newcommand{\defref}[1]{Definition \ref{#1}}

\newcommand{\mysecondproof}[1]{}

\newcommand{\bsp}{B} 
\newcommand{\dg}{\delta}

\newcommand{\DG}{DG}

\newcommand{\convol}{\ast}
\newtheorem{corollary}{Corollary}[section]
\newtheorem{theorem}{Theorem}[section]
\newtheorem{lemma}{Lemma}[section]
\newtheorem{prop}{Proposition}[section]
\newtheorem{definition}{Definition}[section]
\theoremstyle{remark}
\newtheorem{example}{Example}[section]

\newcommand{\dgout}{u}
\newcommand{\itoj}{i:j}
\newcommand{\gs}{s} 
\newcommand{\tsft}{\kft_{\knts}}  
\newcommand{\divdif}{\mathord
   {\kern.43em\vrule width.6pt height5.6pt depth-.28pt\kern-.43em\Delta}}
\newcommand{\mc}{\mathbf{\cfno{}}}
\newcommand{\R}{\mathbb{R}}

\newcommand{\siac}{SIAC}
\newcommand{\bsiac}{PSIAC}
\newcommand{\proto}{prototype}

\newcommand{\matLI}{M_{0,\bkft,\idxft}}

\newcommand{\bp}{\mathbf{p}} 

\newcommand{\Nbs}[2]{N(#1\,|\,#2)}

\title{Non-uniform Discontinuous Galerkin\\ Filters via Shift and Scale}

\author{Dang-Manh Nguyen and J\"org Peters}

\begin{document}
\maketitle 
\begin{abstract}
Convolving the output of Discontinuous Galerkin (\DG) computations
with symmetric Smoothness-Increasing Accuracy-Conserving
(SIAC) filters can improve both smoothness and accuracy. 
To extend convolution to the boundaries, several
one-sided spline 
filters have recently been developed.
We interpret these filters as instances of a general class
of position-dependent spline filters that we abbreviate as \bsiac\ filters. 
These filters may have a non-uniform knot sequence
and may leave out some B-splines of the sequence.
\par 
For general position-dependent filters, we  prove that rational knot sequences 
result in rational filter coefficients.
We derive symbolic expressions for \proto\ 
knot sequences, typically integer sequences that may include
repeated entries and corresponding B-splines, some of which may be skipped. 
Filters for shifted or scaled knot sequences
are easily derived from these prototype filters so that 
a single filter can be re-used in different locations and 
at different scales.
Moreover, the convolution itself reduces to executing a single dot product
making it more stable and efficient than the existing approaches 
based on numerical integration.
The construction is demonstrated for several established
and one new boundary filter.
\end{abstract}


\section{Introduction}\label{sec:intro}

Since the output of Discrete Galerkin (\DG) computations often captures higher
order moments of the true solution \cite{moment78}, postprocessing DG output
by convolution can improve both smoothness and accuracy 
\cite{Bramble77,siac02hyperbolic}. 
In the interior of the domain of computation, symmetric
smoothness increasing accuracy conserving (SIAC) filters 
have been demonstrated to provide optimal accuracy \cite{siac02hyperbolic}.
However the symmetric footprint precludes using these filters
near boundaries of the computational domain. 
 
To address this problem Ryan and Shu \cite{siac2003}
pioneered the use of {one-sided spline filters}. 
The Ryan-Shu filters improve the $L^2$ error,  but not necessarily the   
point-wise errors.
In fact, the filters are observed to increase the pointwise error  
near the boundary, motivating the design of the SRV filter \cite{siac2011},
a filter \kernel\ of increased support.
Due to near-singular calculations, a stable numerical
derivation of the SRV filter requires quadruple precision.
Indeed, the coefficients of the boundary filters 
\cite{siac2003,siac2011,siac2012,siac2014,siac15unifiedView} 
are computed by inverting a matrix whose entries are computed by
Gaussian quadrature; and, as pointed out in \cite{siac2014}, 
SRV filter matrices are close to singular. 
\cite{siac2014} therefore introduced the RLKV filter, 
that augments the filter of \cite{siac2003} by a single additional B-spline.
This improves stability and retains the support-size.
However, RLKV filter errors on canonical test problems are higher 
than those of symmetric filters and they have 
sub-optimal $L^2$ and $L^\infty$ superconvergence rates \cite{siac2014},
as well as poorer derivative approximation \cite{siac16der} than
SRV filters. Consequently, SRV filters 
merit attention if their stablity can be improved. 

The contribution of this paper is to reinterpret the published one-sided
filters in the framework of position-dependent
spline filters with general knot sequences \cite{JPetersFcoefs}
(that were inspired by the partial
symbolic expression of coefficients for uniform 
knot sequences in \cite{siac15unifiedView}).
This reinterpretation allows expressing and pre-solving them in sympbolic
form, preempting instability so these filters can reach their full potential.
Specifically, this paper 
\begin{itemize}
\item[\itmsym]
proves general properties of SIAC filters, when their knots 
are shifted or scaled;
\item[\itmsym]
uses the properties to express the \filter s in a factored, semi-explicit 
form that becomes explicit for given knot patterns and yields
rational coefficients for rational knot sequences; 
\item[\itmsym]
characterizes a class of position-dependent SIAC filters,
short \bsiac\ filters, based on splines with general knot sequences;
the coefficients of \bsiac\ filters are polynomial expressions in the 
position;
\item[\itmsym]
shows that SRV and RLKV 
\cite{siac2003,siac2011,siac2012,siac2014,siac15unifiedView} 
are \bsiac\ filters;
and
\item[\itmsym]
illustrates the general framework by comparing the established numerical
approach with
the new symbolic filter derivation and by adding
a new effective, stably-computed filter
of lower degree than SRV or RLKV.
\end{itemize}
\def\figw{.31\linewidth} 
\begin{figure}[h] \footnotesize
 \centering
 \begin{tabular}{ccc} 
  \includegraphics[width=\figw]{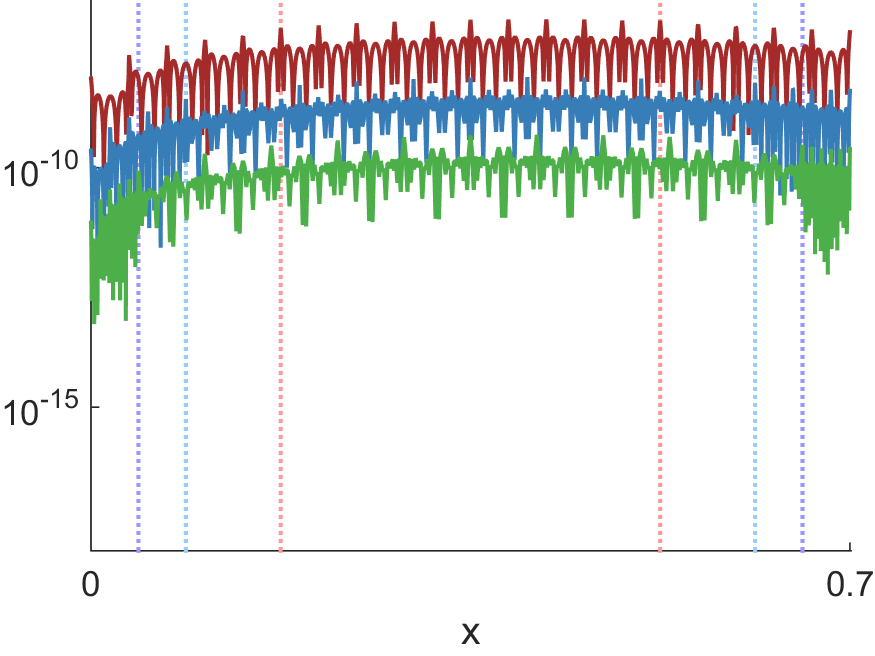}
  &
  \includegraphics[width=\figw]{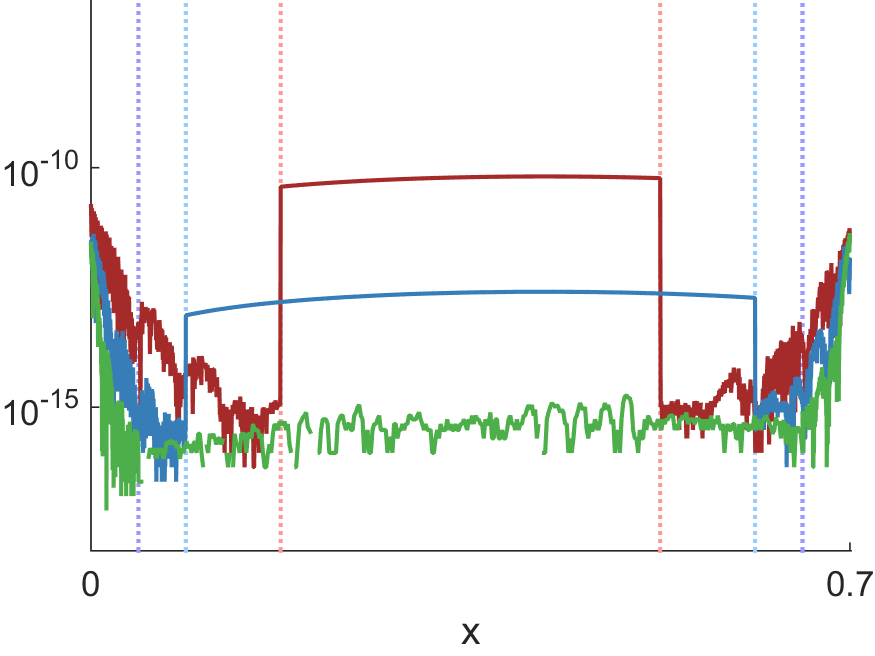}
  &
  \includegraphics[width=\figw]{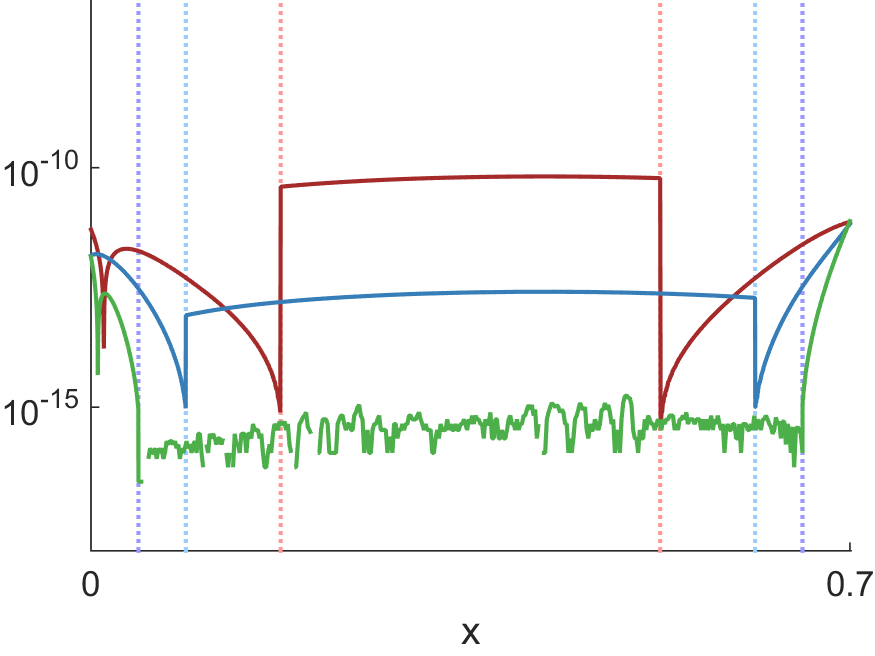}
  \\
(a) \DG\ output error & (b) SRV \oldMethod & (c) SRV \newMethod
\\
  &
  \includegraphics[width=\figw]{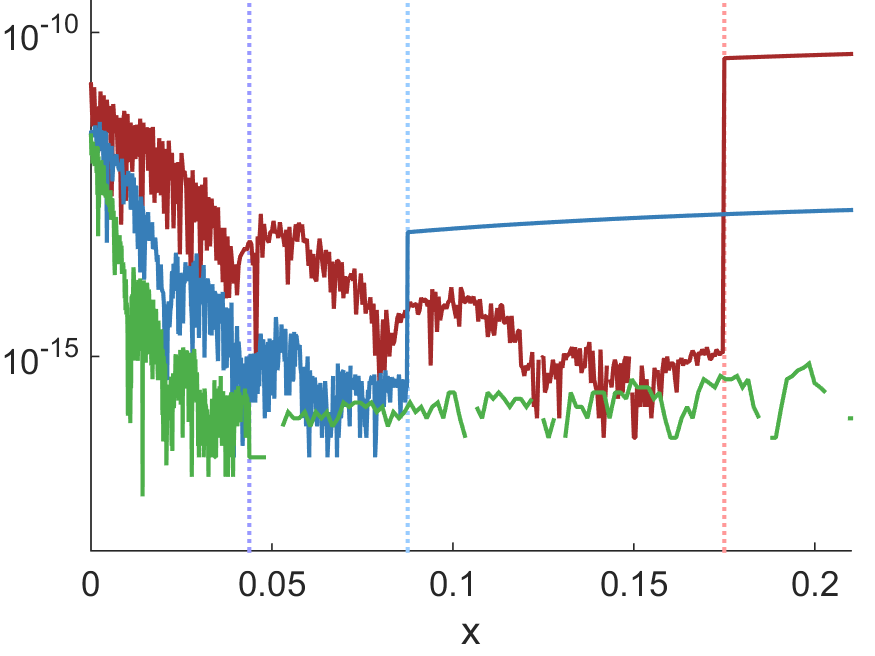}
  &
  \includegraphics[width=\figw]{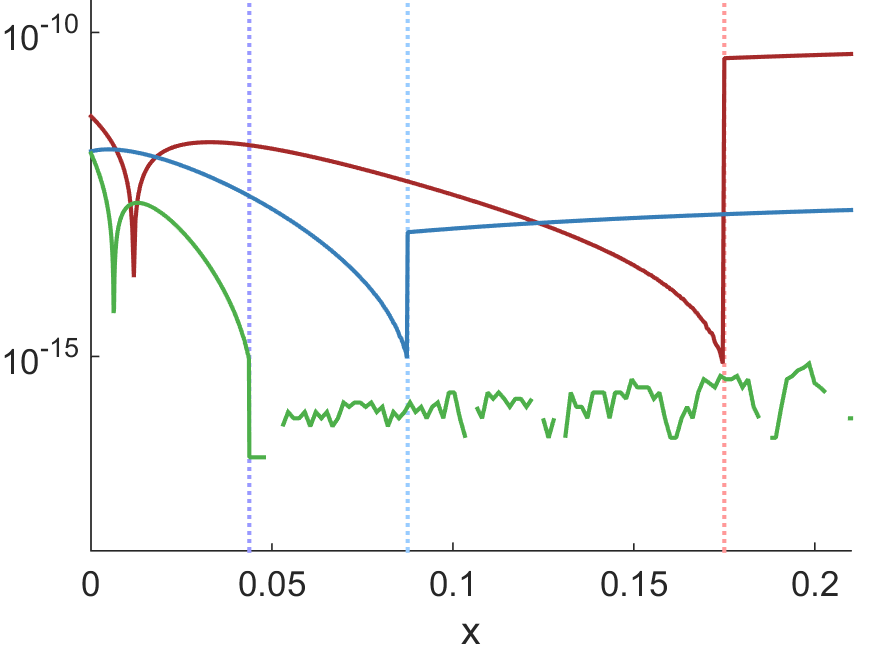}
  \\ 
  &(b') Left zoom of (b) & (c') Left zoom of (c)
 \end{tabular}
\caption{ 
(a) Point-wise error of the $L^2$-projection of 
$\dgout(x,0) = \frac{7}{10}\sin(\pi \sqrt{\frac{10}{7}} x)$
(interpreted as a \DG\ approximation at time $\tm=0$)
onto the space of piecewise \emph{cubic} Bernstein-B\'{e}zier polynomials.
The graphs in each figure correspond, from top to bottom,
to refined intervals separated by $\nkdg=20,40,80$ break points.
(b,c) Point-wise error after double precision convolution with
the filter of \cite{siac2011} computed      
(b,b') based on the \oldMethod\ {\it vs.} 
(c,c') as \bsiac\ filter using the \newMethod. 
The vertical dotted lines separate the boundary region where 
position-dependent one-sided filters are applied from the interior 
where symmetric filters are applied.
}
\label{fig:teasing}
\end{figure} 
The new characterization allows,
in standard double precision, to replace 
the current three-step \emph{\oldMethod} of approximate computation of the matrix,
its inversion and application to the data by Gauss quadrature, 
by a single-step \emph{\newMethod} 
derived in \thmref{thm:symFilteredDG}.
Fig.~\ref{fig:teasing} contrasts, for standard double precision,
the noisy error of numerical SRV filtering
with the error of the new symbolic formulation.

Not only is the \newMethod\ more stable, but  
\begin{itemize}
\item[\itmsym]
scaled and shifted version of 
the filter are easily obtained from one symbolic prototype filter;
\item[\itmsym]
computation is more efficient:
filtering the \DG\ output reduces to a single dot product of
two vectors of small size;  
\item[\itmsym]
stabilizes and simplifies computing  
approximate derivatives of the filtered DG output \cite{siac05der}
(see also \cite{thomee77der,siac09der,siac16der})  
to differentiating the polynomial representation of the 
filtered output; and
\item[\itmsym] 
proves the smoothness of the so-filtered DG output to be $C^\infty$.
\end{itemize}
 
The last point is of interest since \cite{siac2003,siac2011} 
observed and conjectured that the smoothness of the filtered \DG\ computation
is the same as the smoothness in the domain interior where the
symmetric filter applied. Not only does this conjecture hold true, but
\thmref{thm:symFilteredDG} implies that the smoothness is $C^\infty$.

\paragraph{Organization.}  
\secref{sec:backgound} introduces the canonical test equation,
B-splines, convolution, and a generalization of the formula 
of \cite{JPetersFcoefs}
for convolution with splines based on arbitrary knot sequences.
\secref{sec:overShiftedKnots} reformulates the  generalized
convolution formula and 
derives the convolution coefficients when the
knot sequence is a scaled or shifted copy of a \proto\ sequence.
\secref{sec:filtered}, in particular
\thmref{thm:symFilteredDG}, summarizes the resulting efficient
convolution with \bsiac\ filters based on scaled and shifted
copies of a \proto\ sequence.
\secref{sec:someFilters} shows that SRV and RLKV are \bsiac\ filters
and advertises their true potential by 
comparing their computation using the 
\oldMethod\ to using the more stable \newMethod.
To illustrate the generality of the setup,
a new multiple-knot linear filter is added to the picture.

\section{Background} 
\label{sec:backgound}
Here we establish the notation, distinguishing between filters and \DG\ output,
explain the canonical test problem, the \DG\ method, B-splines
and reproducing filters and we review one-sided and position-dependent
\siac\ filters in the literature.

\subsection{Notation}
Here and in the following, we abbreviate the sequences
\begin{align*}
i:j  \,\, &:=\, \, 
\begin{cases}
   (i,i+1,\ldots,j-1,j), \text{if } i\le j, \\
   (i,i-1,\ldots,j+1,j), \text{if } i > j,
\end{cases}
\quad
   s_{i:j} := (s_i,\ldots,s_j).
\end{align*}
We denote by $f\convol g$ the convolution of a function $f$ with
a function $g$, i.e.\ 
\begin{align*}
   (f\convol g)(x) := \int_\R f(t) \, g (x-t) \, \drm t
   = (g\convol f)(x),
\end{align*}
for every $x$ where the integral exists
and we use the terms \emph{filter}, \emph{filter kernel} and \emph{kernel}
interchangeably, even though, strictly speaking, filtering means
convolving a function with a kernel.
We reserve the following symbols:
\begin{align*}
\begin{tabular}{l l}
$\dDG$ & degree of the \DG\ output; \\
$\nkdg$ & number of intervals of the \DG\ output; 
\\
$s_{0:\nkdg}$ & \emph{prototype} increasing break point sequence, 
typically integers;
\\
& the break sequence of the \DG\ output is $\hh  s_{0:\nkdg}$;\\
$\df$ & degree of the filter kernel;  \\
$\nc+1$ & number of filter coefficients  \\
& for reproduction of polynomials up to degree $\nc$;
\\
$\idxft := (\idxfst,\ldots,\idxlst)$ & index sequence;\\
&
if the B-splines of the filter are consecutive, then $\idxlst = \nc$;
\\
$\nkft$& number of knot intervals spanned by the filter;\\
&
$\nkft = \idxlst+\df+1$; 
\\
$\bkft := \kft_{0:\nkft}$ & prototype (integer) knot 
sequence of the filter;\\
        & the input knot sequence of the filter is $\hh \kft_{0:\nkft} + \xi$
        \\
        &where $\xi$ is the shift and $\hh$ scales.
\end{tabular}
\end{align*}
The notation is illustrated by the following example.
\begin{example} \label{ex:symbols}
Consider a degree-one spline filter, i.e.~$k=1$, defined over 
the knot sequence $\bkft := 0:6$ and associated with the index set 
$\idxft := \{0,3,4\}$. 
That is, the two B-splines defined over the knot squences $1:3$ and $2:4$ are 
omitted, $\nkft=6$, $r=2$ and $\idxlst=4$.
A linear ($d=1$) \DG\ output on
200 uniform segments
of the interval $[-1..1]$ implies
$\nkdg=200$, $h = \frac{1}{100}$ and $s_{0:\nkdg} = -100:100$.
\end{example}


\subsection{The canonical test problem, the Discontinuous Galerkin method
and B-splines}
To demonstrate the performance of the filters on a concrete example,
\cite{siac2003} used the following univariate hyperbolic 
partial differential wave equation:
 \begin{align} 
  \frac{du}{d\tm} +
  \frac{d}{dx}\Big( \kappa(x,\tm) \, u\Big) &= \rho(x,\tm), &x \in (a,b), 
  \tm \in (0,\tmend)
  \label{eq:hypEqs}
  \\
  u(x,0) &= u_0(x), &x\in
[a,b] \notag
 \end{align} 
subject to periodic boundary conditions, $u(a,\tm) = u(b,\tm)$, or Dirichlet
boundary conditions $u(e,\tm) = u_0(\tm)$ where, depending on the sign of
$\kappa(x,\tm)$, $e$ is either $a$ or $b$.
Subsequent work 
\cite{siac2003,siac2011,siac2014}
adopted the same differential equation to test their new one-sided filters 
and to compare to the earlier work. Eq.~\eqref{eq:hypEqs} is therefore
considered the \emph{canonical test problem}. 
We note, however, that \siac\ filters apply more widely, for example
to FEM and elliptic equations \cite{Bramble77}.
\par 
In the \DG\ method,
the domain $[a..b]$ is partitioned into intervals by a sequence
$\hh s_{0:\nkdg}$ of break points $a=:\hh s_0,\ldots,\hh s_\nkdg:=b$.
Assuming that the sequence is rational, introducing $\hh$ will allow
us later to consider the prototype sequence $s_{0:\nkdg}$ of integers.
The \DG\ method approximates the time-dependent solution
to Eq.~\eqref{eq:hypEqs} by
\begin{equation} \label{eq:DGoutput} 
   \dgout(x,\tm)  
   :=
   \sum_{i=0}^m \, \dgout_i(\tm) \, \Nbf{i}{x}{\hh s_{0:\nkdg}},
\end{equation}
where the scalar-valued functions $\Nbf{i}{.}{\hh s_{0:\nkdg}}$,
$0 \leq i \leq m$, are linearly independent and satisfy the scaling relations
\begin{equation} \label{eq:DGbasisfuns}
 \Nbf{i}{hx}{hs_{0:\nkdg}} = \Nbf{i}{x}{s_{0:\nkdg}}.
\end{equation}
Examples of functions $\Nbfi{i}$ 
are non-uniform B-splines \cite{deboor}, 
Chebysev polynomials and Lagrange polynomials.
Relation \eqref{eq:DGbasisfuns} is typically used 
for refinement 
in 
FEM, DG or Iso-geometric PDE solvers.

Setting $v(x) := \Nbf{j}{x}{\hh s_{0:\nkdg}}$,
the weak form of Eq.~\eqref{eq:hypEqs}
for $j=0,1,\ldots,m$,
\begin{equation}
 \int_a^b  \big( \frac{du}{d\tm} \, v - \kappa(x,\tm) u \, \frac{dv}{dx} \big)
 \, \drm x = 
 \int_a^b \rho(x,\tm)\,v\,\drm x -
 \Big(\kappa(x,\tm)u(x,\tm)v(x)\Big)\Big|_{x=a}^{x=b},
\end{equation} 
forms a system of ordinary differential equations with 
the coefficients $u_i(\tm)$, $0\leq i \leq m$, as unknowns.

\newcommand{\ddfun}{g}
The goal of \emph{\siac\ filtering} is to smooth $\dgout(\xx,\tm)$ in $\xx$ 
by convolution in $\xx$ with a linear combination of B-splines. 
Typically filtering is applied after the last time step when $\tm = \tmend$.
Following \cite{JPetersFcoefs}, we characterize B-splines in terms
of divided differences
\cite{Curry66}. For a sufficiently smooth univariate real-valued function 
$\ddfun$ with $k$th derivative $\ddfun^{(k)}$, divided differences are defined by
\begin{align}
   \divdif(t_i) \ddfun &:= \ddfun(t_i),
   \hskip0.25\linewidth 
   \text{ and for }
   j > i 
   \notag
   \\
   \divdif(t_{i:j}) \ddfun &:= 
   \begin{cases} 
      (\divdif(t_{i+1:j}) \ddfun - \divdif(t_{i:j-1}) \ddfun)/ (t_j-t_i), & 
      \text{ if } t_i\ne t_j,\\
      \frac{1}{(j-i)!} \, \ddfun^{(j-i)}(t_i), &
      \text{ if } t_i = t_j.
   \end{cases}
    \label{eq:divdif}
\end{align}
If $t_{\itoj}$ is a non-decreasing sequence, we call its elements
$t_\ell$ \emph{knots} and
the classical definition
of the \emph{B-spline of degree $\df$
with knot sequence $t_{\itoj}, \, j:=i+\df+1$} is
\begin{align}
   \bsp(x|t_{\itoj}) &:= (t_{j}-t_i)\ \divdif(t_{\itoj})
   (\max\{(\cdot-x),0\})^{\df}.
   \label{eq:defa}
\end{align}
Here $\divdif(t_{\itoj})$ acts on the function 
$\ddfun: t \to (\max\{(t-x),0\})^{\df}$ for a given $x \in \R$.
Consequently, a B-spline is a non-negative piecewise polynomial function in $x$
with support on the interval $[t_i..t_j)$.
If $\mult$ is the multiplicity of the number $t_\ell$ in the sequence $t_{\itoj}$,
then $\bsp(x|t_{\itoj})$ is at least $\df-\mult$ times continuously differentiable
at $t_\ell$.
This definition of $\Bsp{t}{t_{i:i+\df+1}}$
agrees, after scaling,
with the definition $\Nbs{t}{t_{i:i+\df+1}}$ of the B-spline by recurrence
\cite{Boor:2002:BB}:
\begin{equation} 
  \Nbs{t}{t_{i:i+\df+1}} = \frac{t_{i+\df+1}-t_i}{\df+1}\, 
  \Bsp{t}{t_{i:i+\df+1}}.
\end{equation}

\subsection{\siac\ filter \kernel\ coefficients}
\label{subsec:defM0}

A piecewise polynomial $\sk: \R\to\R$ is said to 
be a \emph{\siac\ spline \kernel\ of \orderft\ $\nc$}
if convolution of $\sk$ with monomials reproduces the monomials up
to degree $\nc$, i.e., if 
\begin{align}
   (\sk \convol (\cdot)^\dg) (x) = x^\dg, \qquad \dg = 0..\nc.
   \label{eq:convol}
\end{align} 

Mirzargar et al. \cite{siac15unifiedView} derived
semi-explicit formulas when the  filter has uniform knots
while \cite{JPetersFcoefs} gives semi-explicit formulas
for the coefficients of spline kernels over general knot sequences. 
The following definition further generalizes these formulas by allowing to skip
some B-splines when constructing the \filter.

\begin{definition} [\siac\ spline filter \kernel]
Let $\idxft := (\idxfst,\ldots,\idxlst)$ be a sequence of 
strictly increasing integers between $0$ and $\idxlst$.
A \emph{\siac\ spline \filter} of degree $\df$ and \orderft{ }$\nc$ 
with index sequence $\idxft$ and knot sequence $\kft_{0:\nkft}$ 
is a spline
\begin{equation*}
   \sk(x) := \sum_{\sft \in \idxft}  \cfno{\sft} B(x|\tsft),
\end{equation*}
of degree $\df$ with coefficients $\cfno{\sft}$ chosen so that
\begin{align}
   &\Bigl( \sum_{\sft\in\idxft}  \cfno{\sft} B(\cdot|\tsft)
   \convol (-\cdot)^\dg \Bigr) (x)
   =
   (-x)^\dg, \qquad \dg = 0,\ldots,\nc.
   \tag{\ref{eq:convol}'}
\end{align}
\label{def:siackernel} 
\end{definition} 
\noindent
When $\idxft \, = \, 0:\nc$ then \defref{def:siackernel} replicates the
definition of \cite{JPetersFcoefs}. 
\begin{lemma}[\siac\ coefficients]
The vector
$\mc := [\cfno{0},\ldots,\cfno{\nc}]^\tr \in \R^{\nc+1}$ 
of B-spline coefficients of the \siac\ filter with
index sequence $\idxft := (\idxfst,\ldots,\idxlst)$  and
knot sequence $\bkft := \kft_{0:\nkft}$ is 
 \begin{align}
    \mc := \text{first column of }
    M_{0,\bkft,\idxft}^{-1},
    \qquad  
   M_{0,\bkft,\idxft} :=
    \left[ \begin{matrix}
    \divdif{\tsft} \xx^{\df+1+\dg} 
    \end{matrix} \right]_{\dg=0:\nc, \,\sft\in\idxft}.
    \label{eq:M}
\end{align}
\label{thm:siac}
\end{lemma}
\renewcommand{\tsft}{\kft_{\sft:\sft+\kp}}
\newcommand{\ssftm}{\kft^1_{\sft:\sft+\kp-1}}
\newcommand{\ssft}{\kft^1_{\sft:\sft+\kp}}
\newcommand{\tsftp}{\kft_{\sft+1:\sft+\kp+1}}
\newcommand{\ssftp}{\kft^1_{\sft+1:\sft+\kp}}
\newcommand{\kp}{\kappa}
\begin{proof}
We prove invertibility of $M_{0,\bkft,\idxft}$.
The remaining claims of the lemma then follow as in \cite{JPetersFcoefs}. 

Let $\ell$ be a left null vector of $\matLI$,
i.e.\ for all $\sft \in \idxft$,  $\kp := \df+1$ and 
$p(\xx) := \xx^{\kp} \sum_\dg \ell_\dg \xx^{\dg}$
\begin{equation}
   0 = \divdif{\tsft} \sum_{\dg=0}^\nc \ell_\dg \xx^{\kp+\dg}
   = \divdif{\tsft} \xx^{\kp} \sum_{\dg=0}^\nc \ell_\dg \xx^{\dg} 
   = \divdif{\tsft} p(\xx).
   \label{eq:nullp}
\end{equation}
Let $q$ be the interpolant of $p$ at $\kft_{\sft:\sft+\kp+1}$,
i.e.\ spanning two consecutive hence overlapping knot sequences.
If knots repeat, $q$ is a Hermite interpolant.
By Rolle's theorem, the derivative 
$D(p-q)$, vanishes at a set of knots $\ssft$ interlaced with
$\kft_{\sft:\sft+\kp+1}$, i.e.\ $\kft_i\le \kft^1_i\le \kft_{i+1}$.
The inequality is strict unless $\kft_i = \kft^1_i$ represents a multiple root.
Then $Dq$ (Hermite) interpolates $Dp$ at $\ssft$ and by the relation
between divided difference and derivatives,
\begin{align*}
   \kp\divdif{\tsft} p = \divdif{\ssftm} Dp 
   &\text{ and }
   \kp\divdif{\tsftp} p = \divdif{\ssftp} Dp.
\end{align*}
Induction yields 
\begin{align*}
\kp! \divdif{\tsft} p 
   = 2\divdif{\kft^{\kp-1}_{\sft:\sft+1}} D^{\kp-1}p,
   &\text{ and }
\kp! \divdif{\tsftp} p 
   = 2\divdif{\kft^{\kp-1}_{\sft+1:\sft+2}} D^{\kp-1}p
   \\
\text{ and finally } 
\kp! \divdif{\tsft} p = D^{\kp}p(\kft^\kp_\sft),\ 
   &\text{ and }
\kp! \divdif{\tsftp} p = D^{\kp}p(\kft^\kp_{\sft+1}) 
\end{align*}
for $\kft^{\kp}_\sft \le \kft^{\kp}_{\sft+1}$.
That is,
the $\kp$th divided difference of each sequence equals
the $\kp$th derivative at points $\kft^{\kp}_\sft$
respectively $\kft^{\kp}_{\sft+1}$;
and the shift of the subsequence of knots
from $\tsft$ to $\tsftp$ implies that 
where $\kft^{\kp}_{\sft+1}$ is either strictly to the right of 
$\kft^{\kp}_\sft$ or $\kft^{\kp}_{\sft}$ is a multiple root. 
Then \eqref{eq:nullp} implies  
that $D^{\kp} p$, a polynomial of degree at most $\nc$,
has $\nc+1$
roots counting multiplicity, hence is the zero polynomial.
Given the factor $t^\kp$ of $p$, this can only hold if $\ell=0$, 
i.e.~$M_{0,\bkft,\idxft}$ has no non-trivial left null vector
and, as a square matrix, $M_{0,\bkft,\idxft}$ is invertible. 

Since all sequences $\tsft$, $\sft \in \idxft$ can be obtained
by repeated shifts to the right, the conclusion $\ell=0$ holds for 
$\sft\in \idxft$ and $M^{-1}_{0,\bkft,\idxft}$ is well-defined.
%
%
%
%
%
\end{proof}
\renewcommand{\tsft}{\kft_{\sft:\sft+\df+1}}

\mysecondproof{
\begin{proof}
Applying Peano's formula, 
$
   \frac{1}{k!} \int_\R \bsp(t|t_{0:k}) g^{(k)}(t) dt
   =
   \divdif(t_{0:k}) g 
$
to $g(t) = \kft^{\df+1+\dg}$, we see that invertibility
of $M_{0,\bkft,\idxft}$ is equivalent to 
\begin{align}
   \int \Bsp{s}{\kft_{sft:sft+df+1}} p(s) \drm s &= 0
   \text{ for all } p \in \Pi_\nc,
   \text{ for all } \sft \in \idxft 
   \label{eq:invert}
   \\
   \notag
   \text{ implies }
   p=0
\end{align}
Since $\Bsp{s}{\kft_{sft:sft+df+1}}\ge 0$, $p$ must be zero or change sign 
so the integral is zero. 
This implies a root in $(\kft_{sft}..\kft_{sft+df+1})$.
The right-shifted neighbor $\Bsp{s}{\kft_{sft+\gamma:sft+df+1+\gamma}}$,
however, requires a different root further to the right to 
set its integral to zero. The resulting $\nc+1$
distict roots imply \eqref{eq:invert}.
\end{proof}
}

\subsection{Review of Symmetric and Boundary SIAC filters}
\label{sec:intro-osFt}
We split the \DG\ data at any known discontinuities
and treat the domains separately.
Then convolution can be applied throughout a given closed
interval $[\lft..\rgt]$. 

A SIAC spline \kernel\ with knot sequence $t_{0:\nc+\df+1}$ 
is \emph{symmetric} (about the origin in $\R$)  if
$t_\ell + t_{\nc+\df+1-\ell}=0$ for $\ell=0: \lceil (\nc+\df+1)/2 \rceil$.
Convolution with a symmetric SIAC \kernel\ of a function $g$ at $x$ 
then requires 
$g$ to be defined in a two-sided neighborhood of $x$. 
Near boundaries, Ryan and Shu \cite{siac2003} therefore suggested 
convolving the \DG\ output with a  one-sided \kernel\
whose support is shifted to one side of the origin:
for $x$ near the left domain endpoint $\lft$,
the one-sided SIAC \kernel\ is defined over
$ (x-\lft) + h \big( -(3d+1):0 \big)$
where $\ddg$ is the degree of the \DG\ output.
The Ryan-Shu $\xx$-position-dependent 
one-sided kernel yields optimal $L^2$-convergence,
but its point-wise error near $\lft$ can be larger than that of the
\DG\ output.

\par 
In \cite{siac2011}, Ryan et al.\
improved the one-sided \filter\ by increasing its monomial reproduction
from degree $\drepro = 2\ddg$ to degree $\drepro = 4\ddg$.
This one-sided \kernel\ reduces the boundary error when $\dDG=1$
but the \kernel\ support is increased by $2\ddg$ additional knot intervals
and 
numerical roundoff requires high precision calculations to determine
the \kernel's coefficients.
(\cite{siac2011} did not draw conclusions for degrees $\dDG>1$.)
\par
Ryan-Li-Kirby-Vuik \cite{siac2014} suggested 
an alternative position-dependent 
one-sided \filter\ that has the same support size
as the symmetric kernel and its reproduction degree is only higher by one. 
The new idea is that the spline space defining the \filter\
is enriched by one B-spline.
The new \filter\ computation is
stable up to degree $\dDG=4$ in double precision.
When $\dDG=1$ the RLKV \filter's point-wise error 
on the canonical test problem is as low as that
of the symmetric SIAC \filter\ that is applied in the interior.
However, when $\dDG>1$, the RLKV error is higher than
that of the symmetric \filter. 
\par 
In 
\cite{siac2003,siac2011,siac2012,siac2014} 
convolution with position-dependent one-sided \filter s 
is computed as follows.
For each domain position $\xx$,
\begin{itemize}
\item[\itmsym] \emph{Calculate \filter\ coefficients}: \\
compute the entries of the 
position-dependent reproduction matrix $M$ by Gaussian quadrature;
solve a corresponding linear system $M\mc=\bp$ for the kernel coefficients
$\mc$ to match the monomials to be reproduced, collected in the vector $\bp$. 
As pointed out in \cite{siac2014}, $M$ may be close to singular
(for example, when $\ddg\geq3$ for the SRV \filter s)
so that higher numerical precision (e.g.\ quadruple precision) is required
to assemble and solve the linear system.
\item[\itmsym] {\it Convolve the \filter\ with the 
\DG\ output} by Gaussian quadrature. 
\end{itemize}
Note that, unlike the (position-independent)
classical symmetric SIAC filter, the position-dependent boundary
kernel coefficients have to be determined afresh for each point $\xx$.
\section{Coefficients of shifted and scaled {\ft}s}
\label{sec:overShiftedKnots}
\newcommand{\dft}{\df}
To calculate filters for \DG\ output more
efficiently and more stably, we reformulate \lemref{thm:siac} in
\emph{multi-index notation}:
\begin{align*}
   \kft^\midx_{0:\nkft} 
   &:=
   \kft_0^{\ix_0} \ldots \kft_{\nkft}^{\ix_{\nkft}}
   \quad \text{and}\quad
   |\kft_{0:\nkft}| := \sum^{\nkft}_{j=0} |\kft_j|
\end{align*}
as follows.

\begin{lemma}  [\siac\ reproduction matrix]
The matrix Eq. \eqref{eq:M} has the alternative form
\begin{align}
   M := M_{0,\kft_{0:\nkft},\idxft}
   &= \left[ \begin{matrix}
    \sum\limits_{|\midx|=\dg} 
    \kft_{\sft:\sft+\dft+1}^{\midx}   
   \end{matrix} \right]_{\dg=0:\nc,\,\sft\in\idxft}.
   \label{eq:M0powerForm}         
\end{align}
\label{lem:M0}
\end{lemma}
\begin{proof}
Applying Steffensen's formula 
\cite[Eq.~(27)]{deboor}\footnote{The index $\alpha$ 
in \cite[Eq.~(27)]{deboor} is misprinted. It should be: $|\alpha| = n-k+1$.},
the divided differences of monomials in \eqref{eq:M} can be rewritten
as \eqref{eq:M0powerForm}.
\end{proof}

Denote the reproduction matrix associated with the shifted knot sequence 
$
   \kft_{\sft:\sft+\dft+1} + \xi,
   \, \xi \in \mathbb{R}
$ as
\begin{equation}
   M_\xi := \left[ \begin{matrix}          
          \sum\limits_{|\midx| = \delta}
          (\kft_{\sft:\sft+\dft+1} + \xi)^{\midx}
         \end{matrix} \right]_{\dg=0:\nc,\,\sft\in\idxft}.
\label{eq:Mxi}
\end{equation}
Since each entry $M_\xi(\dg,\sft)$ is a polynomial of degree $\dg$
in $\xi$, the determinant of $M_\xi$ is a polynomial of degree
$\nc(\nc+1)/2$ in $\xi$.  Using for example Cramer's rule, 
the entries of $M^{-1}_\xi$ are rational functions in $\xi$ whose 
nominator and denominator are polynomials of degree
$\nc(\nc+1)/2$ in $\xi$.
Since the convolution coefficients are the entries of the 
first column of $M^{-1}_\xi$, it is remarkable, that we can show that
the coefficients are not rational but polynomial and of 
degree $\nc$ rather than $\nc(\nc+1)/2$ in $\xi$.

\par 
To prove this claim, we employ the following technical result 
that generalizes a binomial indentity to multiple indices. 
\begin{lemma} [technical lemma]
\label{lem:binomIden}
We abbreviate
$\binom{\bb}{\ab} := \binom{b_0}{a_0}\ldots \binom{b_{\dft+1}}{a_{\dft+1}}$ and
write $\ab \leq \bb$ to indicate that, for each $k\mathrm{th}$ component,
$a_k \leq b_k$. Then for $\delta > |\ab|$,
\begin{equation} \label{eq:binom2}
 \sum_{\alphab \geq \ab,\, |\alphab| = \dg} \binom{\alphab}{\ab} = 
\binom{\dg+\dft+1}{|\ab|+\dft+1} = \binom{\dg+\dft+1}{\dg-|\ab|}.
\end{equation}
\end{lemma}
\begin{proof}
By the Maclaurin expansion: $\frac{1}{(1-x)^{k+1}} = \sum_{\ell=0}^{\infty} \binom{k+\ell}{k} x^\ell$, $|x|<1$, we see that
\begin{equation} \label{eq:binomIden1}
 \frac{x^k}{ (1-x)^{k+1}}  = \sum_{\ell \geq k} \binom{\ell}{k} x^\ell, \quad |x|<1.
\end{equation}
Applying Eq.~\eqref{eq:binomIden1} to both sides of the identity
\begin{equation}
 x^{\df+1} \, \frac{x^{a_0}}{(1-x)^{a_0+1}} \cdots 
 \frac{x^{a_{\df+1}}}{(1-x)^{a_{\df+1}+1}}
 = \frac{x^{|\ab|+\dft+1}} {(1-x)^{|\ab|+\dft+2}},
\end{equation}
and dividing both sides by $x^{\df+1}$, we see that 
\begin{equation} \label{eq:binom1}
 \sum_{\alphab \geq \ab} \binom{\alphab}{\ab} x^{|\alphab|} = 
 \sum_{\ell \geq |\ab| +\dft +1} \binom{\ell}{|\ab|+\dft+1} x^{\ell-\df-1}.
\end{equation}
Selecting the coefficients of $x^{\dg}$ from both
sides of Eq.~\eqref{eq:binom1} yields Eq.~\eqref{eq:binom2}. 
\end{proof}

\begin{lemma} [Reproduction matrix for shifted knots]
The matrices $M_\xi$ and $M := M_{0,\kft_{0:\nkft},\idxft}$
are related by 
\begin{equation} 
    M_\xi  = P_{\dft,\nc}(\xi) \, M := \left[ 
   \begin{matrix}
   1 \\
   \binom{1+\dft+1}{1} \xi & 1\\
   \vdots & \vdots & \ddots\\
   \binom{\nc +\dft+1}{\nc} \xi^\nc & \cdots & \binom{\nc +\dft+1}{1} \xi& 1
   \end{matrix}
    \right] \, M.
\label{eq:MxiM0}
\end{equation}
\label{lem:Mxi}
\end{lemma}
(In \eqref{eq:MxiM0}, only non-zero entries are shown.)
Note that $P_{\dft,\nc}(1)$ is the result of deleting
the first $\dft+1$ rows and $\dft+1$ columns of the lower
triangular Pascal matrix of order $\nc+\dft+1$. 
\begin{proof}
Abbreviating $\tb_j := t_{j:j+\df+1}$, the entry $(\dg,\sft)$ of $M_\xi$ defined by Eq.~\eqref{eq:Mxi} is
 \begin{align}
  M_\xi(\dg,\sft) &= \sum_{|\alphab|=\dg} \sum_{\boldmath{0} \leq \lb \leq \alphab} \binom{\alphab}{\lb} \, \xi^{|\lb|} \, \tb_\sft^{\alphab - \lb}
  \label{eq:Mxi1} 
  \\
   &= 
   \sum_{|\alphab|=\dg} \sum_{\beta=0}^{\dg} \xi^\beta  \sum_{|\lb| = \beta,\, \boldmath{0} \leq \lb \leq \alphab} \binom{\alphab}{\lb} \, \tb_\sft^{\alphab - \lb}     
     = 
\sum_{\beta=0}^{\dg} \xi^\beta  \sum_{|\ab|=\dg-\beta} \tb_\sft^{\ab} \sum_{|\alphab| = \dg,\, \alphab \geq \ab} \binom{\alphab}{\ab}. 
\notag
 \end{align}
The last equality of Eq.~\eqref{eq:Mxi1} follows by substituting
$\ab = \alphab-\lb$. 
Applying Lemma~\ref{lem:binomIden} to \equaref{eq:Mxi1}
and noting that $|\ab|=\dg-\beta$, we see that
\begin{equation} \label{eq:Mxi2}
 M_\xi(\dg,\sft) = \sum_{\beta=0}^{\dg} \binom{\dg+\dft+1}{\beta} \xi^\beta  
M(\dg-\beta,\sft). 
\end{equation}
\equaref{eq:Mxi2} is the expanded form of \equaref{eq:MxiM0}.
\end{proof}

Next we consider the effect of scaling knots, 
as might be done to
refine a \DG\ computation.

\begin{lemma} [Reproduction matrix for scaled knots]
Let $\diag(\vb)$ denote the square matrix with diagonal $\vb$ and 
zero otherwise. Then for $\bkft := \kft_{0:\nkft}$
\begin{align}
   M_{0,h\bkft,\idxft}^{-1}
   &=
   M_{0,\bkft,\idxft}^{-1}\diag(h^{-(0:\nc)}),
   &\text{ where }
   h^{-(0:\nc)} := [1,h^{-1},\ldots,h^{-\nc}].
   \label{eq:scale}
\end{align}
\label{lem:M0h}
\end{lemma}
\begin{proof}
By \lemref{lem:M0}, multiplying the $(\delta+1)$-th row of
$M_{0,\bkft,\idxft}$ by $h^{\delta}$ yields
$M_{0,h\bkft,\idxft}$, $\delta=0,\ldots,\nc$ and hence 
\begin{equation*}
   M_{0,h\bkft,\idxft}
   = \diag([1,h,\ldots,h^{\nc}]) \, 
   M_{0,\bkft,\idxft}
\end{equation*}
which is  equivalent to \eqref{eq:scale}.
\end{proof}
 
Alltogether, we obtain the following semi-explicit formula for 
the filter coefficients.

\begin{theorem} [Scaled and shifted \siac\ coefficients]
The \siac\ filter coefficients $\cfno{\xi;\ell}$
associated with the knot sequence
$
   h\bkft+ \xi 
$
are polynomials of degree $\nc$ in $\xi$:
\begin{equation} 
   \mc_\xi    
   := [\cfno{\xi,\ell}]_{\ell=0:\nc}    
   = \ M_{0,\bkft,\idxft}^{-1}
   \diag(\left[\begin{matrix}(-1)^\ell \, \binom{\ell+\df+1}{\ell}
   \end{matrix}\right]_{\ell=0:\nc})
   \diag(h^{-(0:\nc)})
   \big( \xi^{0:\nc} \big)^{\tr}.
\label{eq:cxih}
\end{equation}
\label{thm:oneSidedCoefs}
\end{theorem}
\begin{proof}
By Lemma~\ref{lem:Mxi} and Lemma~\ref{lem:M0h}, the matrix $M_\xi$ 
corresponding to the scaled and shifted knot sequence
$ h\bkft + \xi $ has the inverse 
\begin{equation}
  M_\xi^{-1} 
  =
  M^{-1}_{0,\bkft,\idxft} \diag(h^{-(0:\nc)})  (P_{\df,\nc}(\xi))^{-1}.
\label{eq:Mxiinv}
\end{equation}
According to \cite{pascalMatrix}, 
$P_{\dft,\nc}(\xi) = (P_{\dft,\nc}(1))^\xi$ and hence 
  $(P_{\dft,\nc}(\xi))^{-1} = P^{-\xi}_{\dft,\nc}(1) = P_{\dft,\nc}(-\xi)$. 
Since Theorem~\ref{thm:siac} requires only the first column of $M_\xi^{-1}$,
we replace, in \eqref{eq:Mxiinv}, $P_{\dft,\nc}(-\xi)$ by its first column 
and obtain
\begin{equation} 
  \mc_\xi = M^{-1}_{0,\bkft,\idxft}\,
  \diag(h^{-(0:\nc)})
  \,
   \diag(\left[\begin{matrix}(-1)^\ell \, \binom{\ell+\dft+1}{\ell} 
\end{matrix}\right]_{\ell=0:\nc})
\,
\big( \xi^{0:\nc} \big)^{\tr}.
\label{eq:cxih2}
 \end{equation}
 Eq.~\eqref{eq:cxih} follows, because the diagonal matrices commute.
\end{proof}
Compared to \eqref{eq:cxih2} formula \eqref{eq:cxih} has the advantage that
it groups together two matrices that can be pre-computed independent of $h$ and $\xi$.

The following corollary implies that the \kernel\ coefficients 
$\cfno{\xi,\ell}$, can
be computed stably, as scaled integers.
\begin{corollary} [Rational \siac\ filter coefficients $\cfno{\xi,\ell}$]
If the knots $\kft_{0:\nkft}$ are rational, then the 
filter coefficients $\cfno{\xi,\ell}$ 
are polynomials in $\xi$ and $h$ with rational coefficients.
\label{lem:rational}
\end{corollary} 
\begin{proof}
  Lemma \ref{lem:M0} implies that the entries of $M$ are rational
  if the knots are rational. Since the determinant of a matrix with rational
  entries is rational, for example Cramer's rule implies
  that the convolution coefficients are rational.
\end{proof}

\section{Position-dependent (\bsiac) filtering}
\label{sec:filtered}
In this section, we first derive a general factored expression 
for the convolution of \bsiac\ filters with \DG\ data.
Then we specialize the setup to one-sided \bsiac\ filters 
when the \DG\ breakpoint sequence is uniform.

\subsection{Position-dependent (\bsiac) filters} \label{sec:filteredDG:general}

When symmetric SIAC filtering increases the smoothness of the \DG\ 
output, the result is in general a piecemeal function.
One may expect the same of any one-sided kernel.
However, this section proves that convolution with position-dependent
\bsiac-filters yields
a single polynomial piece over their interval of application.
We start by defining position-dependent \kernel s.

\begin{definition} [\bsiac\ kernel] 
A \bsiac\ \kernel\ at position $\xx$ has the form 
\begin{equation} \label{eq:filterOverSK}
   f_\xx(s) 
   := \sum_{j\in\idxft} \, \cf{\xx}{j} \Bsp{s}{\hh t_{j:j+\df+1}+\xx},    
   \quad 
   s \in \hh[\kft_0,\kft_\nkft] + \xx.
\end{equation}
\label{def:bsiac_filter}
\end{definition}

\def\fwg{.3\textwidth}
\def\fw{.3\textwidth}
\begin{figure}[ht!]
\centering
 \begin{tabular}{ccc} 
  \includegraphics[width=\fwg,trim=100 50 100 50,clip]{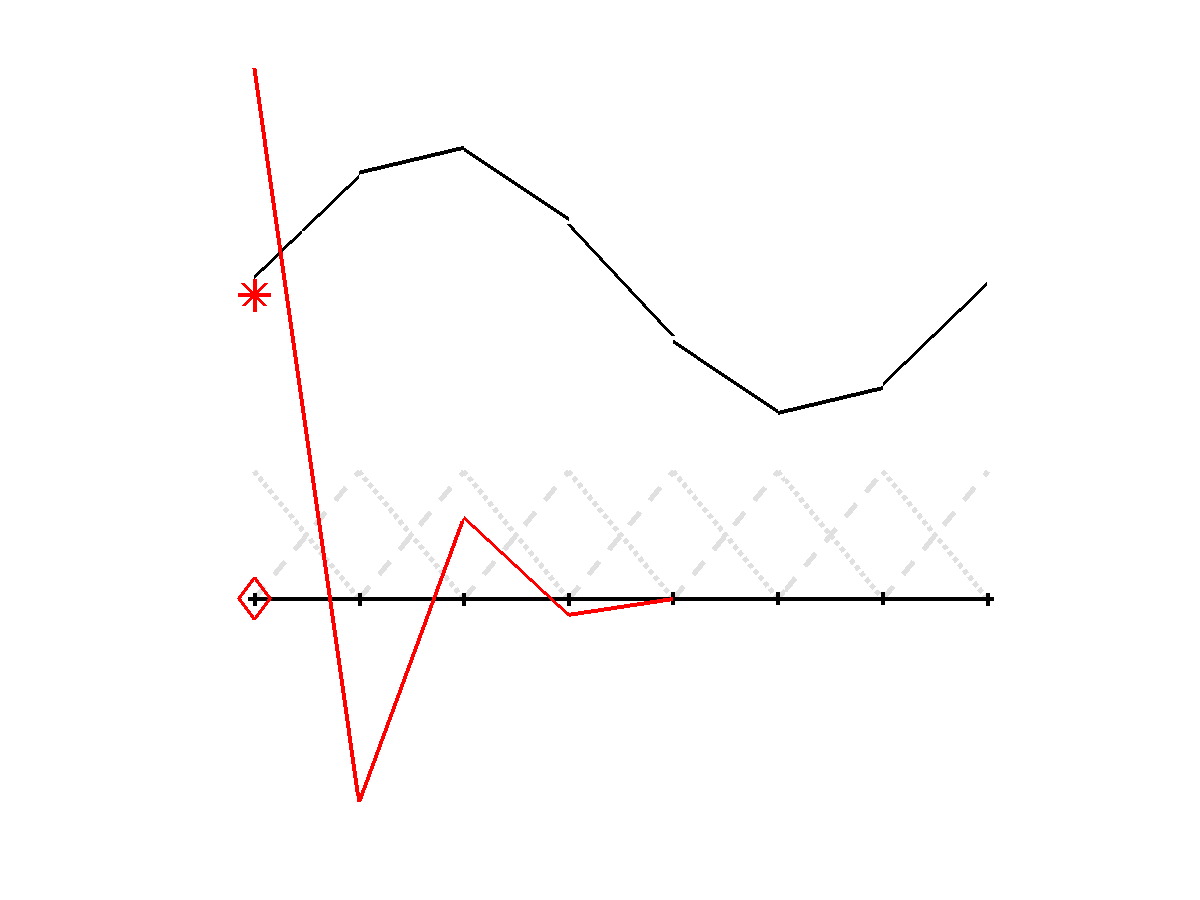}
 &
 \includegraphics[width=\fwg,trim=100 50 100 50,clip]{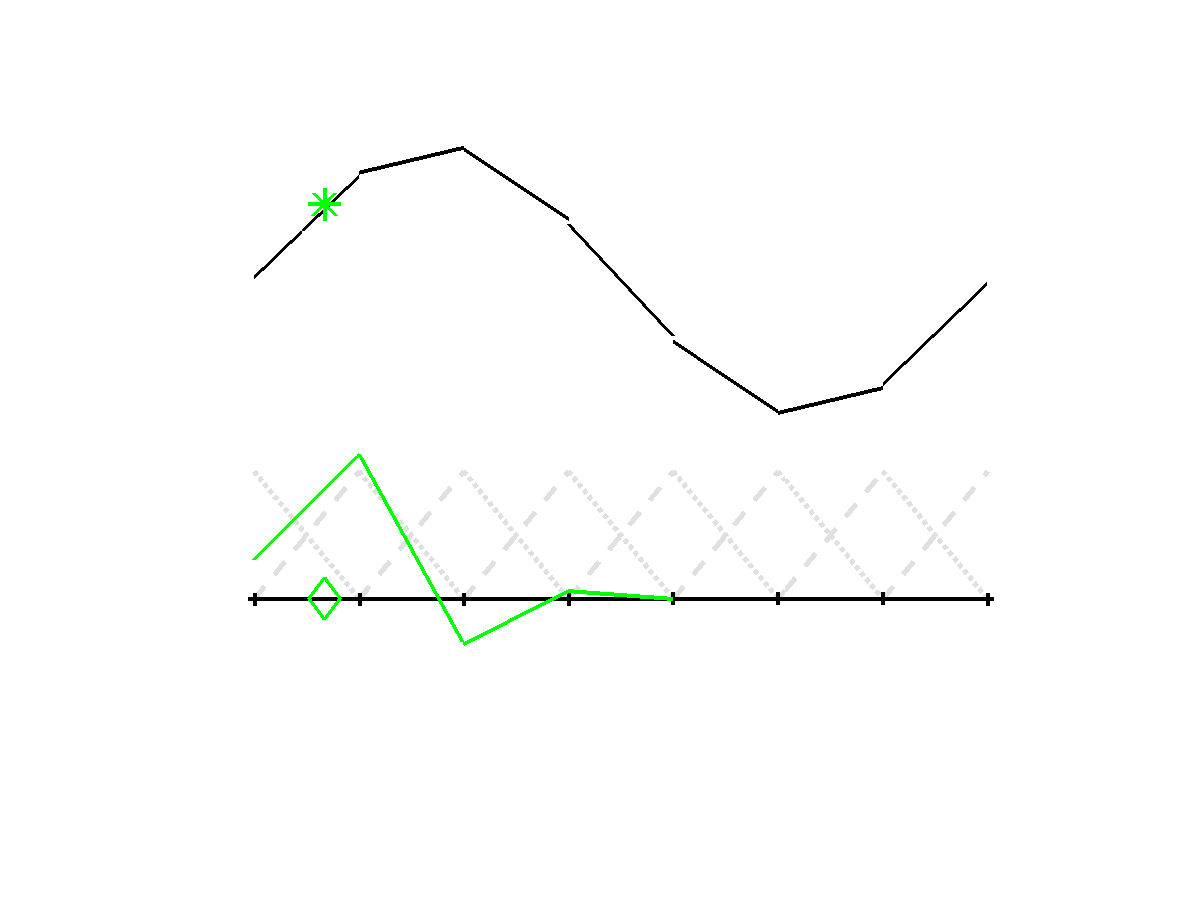}
 &
 \includegraphics[width=\fwg,trim=100 50 100 50,clip]{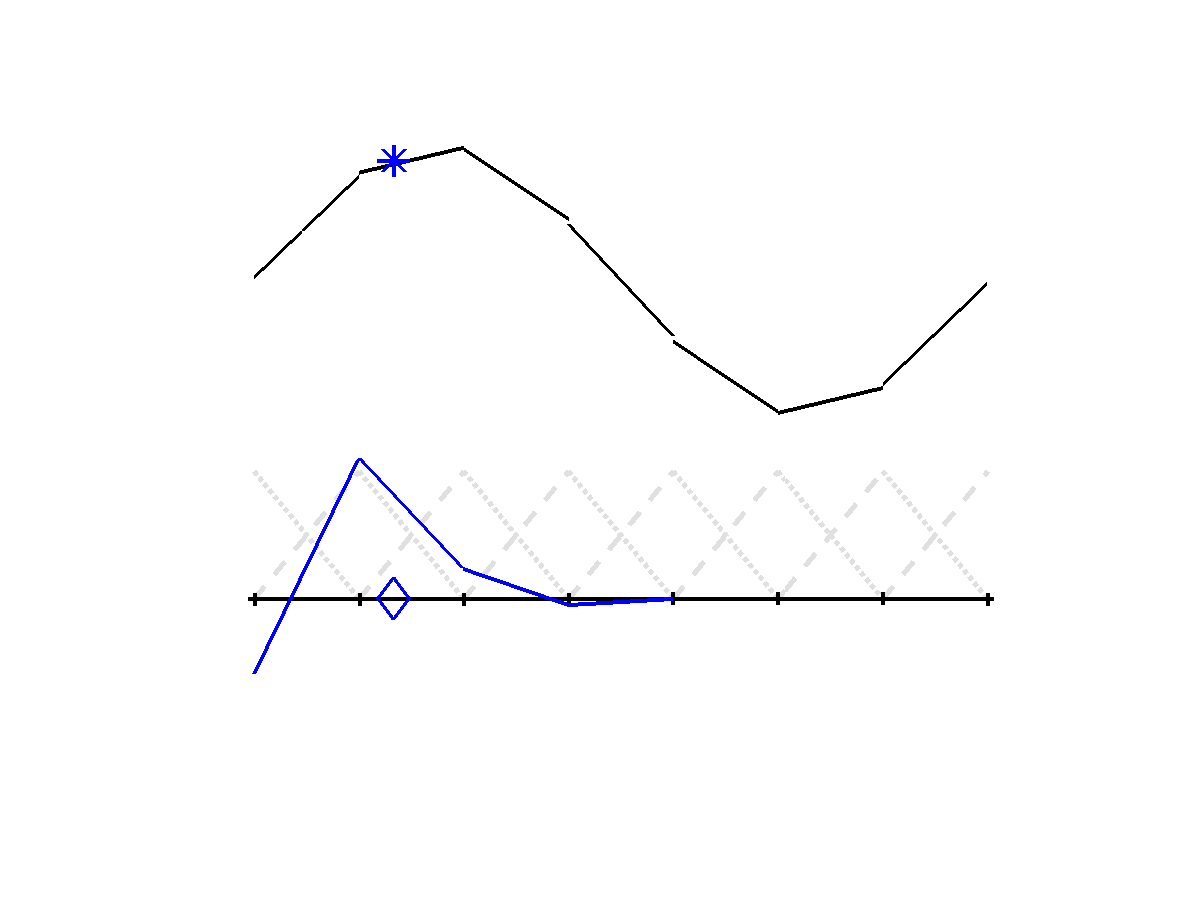}
 \\
 \includegraphics[width=\fw,trim=0 0 0 10,clip]{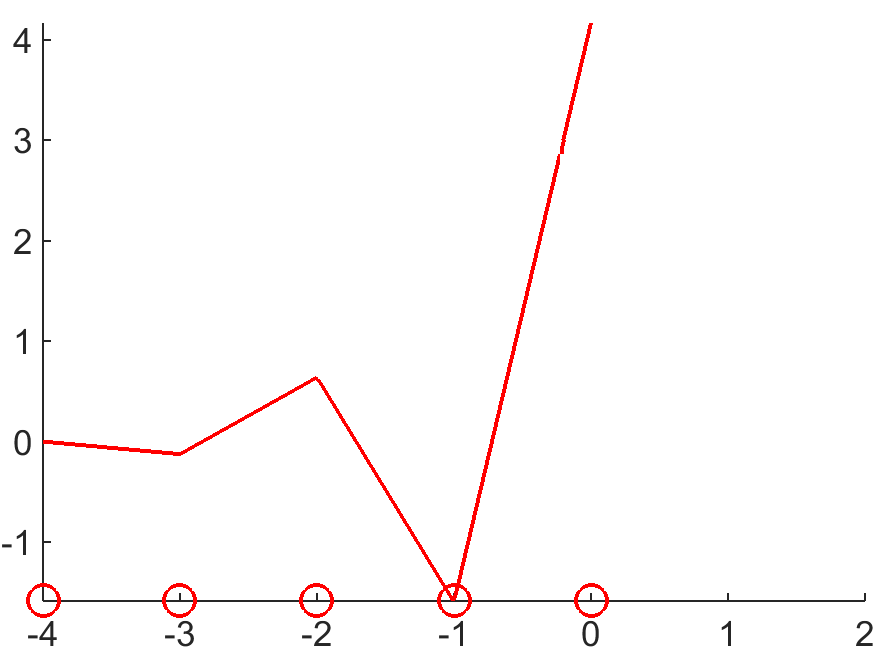}
 &
 \includegraphics[width=\fw,trim=0 0 0 10,clip]{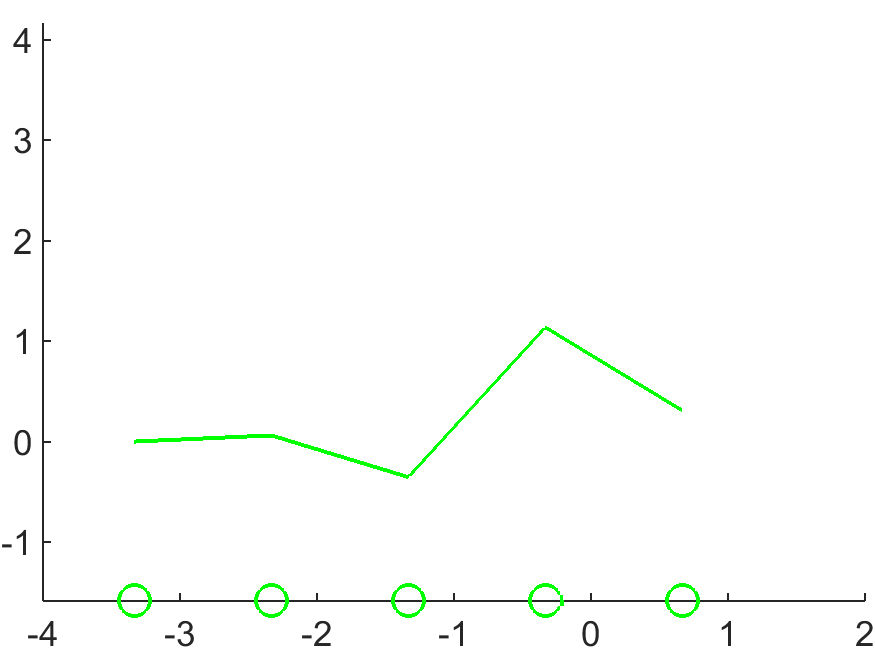}
 &
 \includegraphics[width=\fw,trim=0 0 0 10,clip]{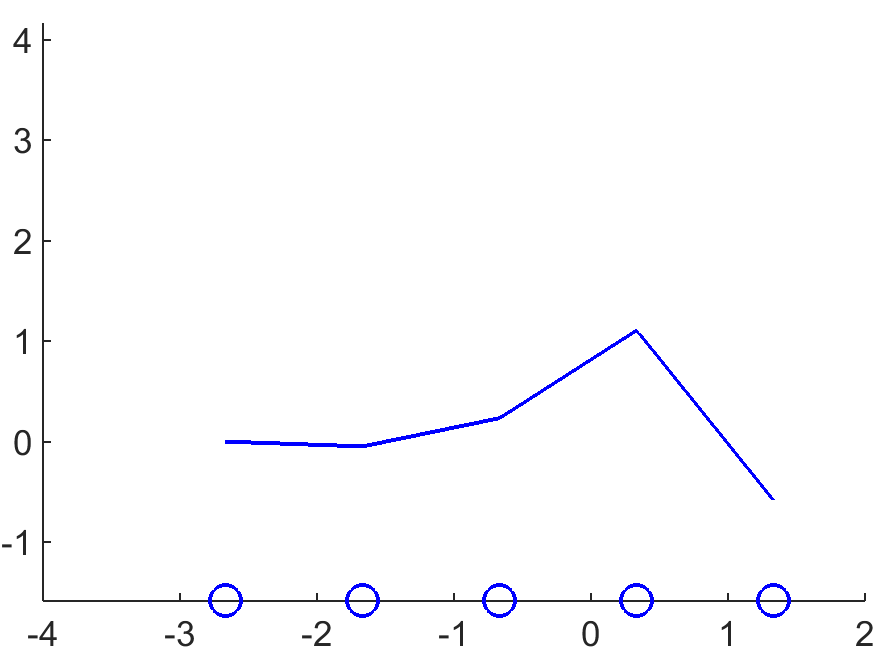}
 \end{tabular}
 \label{fig:psiac}
\caption{Position-dependent filtering at three locations $\xx$
(marked as a diamond) near the left endpoint 
(red diamond; denoted by $a$ in the text) of the piecwise linear
\DG\ output (note the small gaps between the segments).
(Top) The resulting convolved values are marked by a star.
(Bottom) The rightmost knot of the kernel is at location $\xx-a$.}  
\end{figure}

Here we introduced the scaling $\hh$ in anticipation of using a 
\proto\ knot sequence, 
typically
a subset of integers, to create one filter
and then apply its shifts and its scaled version
that can be cheaply computed as explained in
the previous section. That is,
the \DG\ output will be convolved with a \bsiac\ \filter\
$\flt{x-\hh \refx}(s)$ 
of \orderft{ }$\nc$, associated with an index sequence $\idxft$
and defined over the scaled and shifted  knot sequence 
$\hh \kft_{0:\nkft} + x -\hh \refx$ where the constant $\hh \refx$ adjusts
to the left or right boundary.
Substituting $x-\hh \refx$ for $x$ in Eq.~\eqref{eq:filterOverSK}, 
changing variable from $s$ to $x-s$
and recalling that $\nkft=\idxlst+\df+1$, we obtain an alternative spline 
representation of $\flt{x-\hh \refx}(s)$ with $x$-dependent coefficients
$\cfno{\idxlst-j} := \cf{x-\hh \refx}{\idxlst-j}$ 
over the shifted and reversed knot
sequence $\hh (\refx-\kft_{\nkft:0})$: 

\begin{equation} 
   \flt{x-\hh \refx}(s) 
   =
   \sum_{j\in\idxft} \, 
   \cfno{\idxlst-j} \, 
   \Bsp{x-s}{\hh \refx-\hh \kft_{\nkft-j:\idxlst-j}}, 
   \quad 
   s \in \hh[\kft_0,\kft_\nkft] + x -\hh \refx.
   \label{eq:filterOverSK1}
\end{equation}
Now consider
the \DG\ output $u(s,\tm)$ with break point sequence $s_{0:\nkdg}$
as in Eq.~\eqref{eq:DGoutput}.
Convolving $u(s,\tm)$ with $\flt{x-\hh \refx}(s)$ 
yields, after a change of variable $t :=x-s$, the 
\emph{filtered \DG\ output} 
\begin{align}
   \big(u*&\flt{\xx-\hh \refx}\big)(\xx)   \label{eq:convOnesided} \\
   &=\sum_{j\in\idxft} 
   \cfno{\idxlst-j} \, 
   \int_{\hh \kft_0+\xx-\hh \refx}^{\hh \kft_\nkft+\xx-\hh \refx} \dgout(\xx-s,\tm) \, 
   \Bsp{\xx-s}{\hh \refx-\hh \kft_{\nkft-j:\idxlst-j}} \, \drm s \notag
   \\
   &= \sum_{i\in\Ical,\,j\in\idxft}  \dgout_i(\tm) \, 
   \cfno{\idxlst-j} \, 
   \int_{\hh \refx-\hh \kft_\nkft}^{\hh \refx-\hh \kft_0} 
    \Nbf{i}{t}{\hh s_{0:\nkdg}} \, \Bsp{t}{\hh \refx-\hh \kft_{\nkft-j:\idxlst-j}} \, \drm 
t. \notag  
\end{align}
We note that, in the last expression, 
the integral no longer depends on $\xx$.
Rather, the coefficient $\cfno{\idxlst-j} := \cf{x-\hh \refx}{\idxlst-j}$ 
depends on $\xx$: by \thmref{thm:oneSidedCoefs}
$\cfno{\idxlst-j}$ is a polynomial in $\xx$.
This yields the following factored
representation of the convolution.

\begin{theorem} [Efficient \bsiac\ filtering of \DG\ output]
Let $\flt{x}(s)$ be a \bsiac\ \filter\ of \orderft{ } $\nc$
with index sequence $\idxft = (\idxfst,\ldots,\idxlst)$ 
and knot sequence $h\kft_{0:\nkft}+x-h\refx$.
Let 
$
   \dgout(x,\tm) 
   :=
   \sum_{i=0}^m \, \dgout_i(\tm) \, \Nbf{i}{x}{\hh s_{0:\nkdg}},
$
$x \in [a,b]$ and $\tm \geq 0$,
be the \DG\ output.
Let $\Ical$ be the set of indices of basis functions
$\Nbf{i}{.}{\hh \gs_{0:\nkdg}}$ 
with support overlapping
$\hh[\refx-\kft_{\nkft},\refx-\kft_0]$.
Then the \emph{filtered \DG\ approximation}
is a polynomial in $x$
of degree $\nc$:
\begin{align} \label{eq:symFilteredDG}
   \big(u*\flt{x}\big)(x) &= \ub_\Ical 
   \, 
   Q_{\refx}
   \,
   \diag( h^{-(0:\nc)} )
   \,
   \big( (x-\hh\refx)^{0:\nc} \big)^\tr.
   \\
   \ub_\Ical &:= [u_i(\tm)\big]_{i\in\Ical}, 
   \notag
   \\
   Q_{\refx} &:= 
   \ipmatl{}
   \,
	\adia
   \,
   M^{-1}_{0, \bkft,\idxft}
   \,
   \diag(\left[\begin{matrix}(-1)^\ell \, \binom{\ell+\df+1}{\ell} \end{matrix}
   \right]_{\ell=0:\nc}),
   \notag
   \\
   \ipmatl{} &:= \left[ \int_{\refx-\kft_{\lst}}^{\refx-\kft_0} \,
   \Nbf{i}{s}{\gs_{0:\nkdg}}  
   \, \Bsp{s}{\refx-\kft_{\nkft-j:\idxlst-j}} \,
   \drm s \right]_{i\in\Ical,\, j\in\idxft}, 
   \label{eq:Tcinvariant}   
\end{align} 
and $\adia$ the reversal matrix (1 on the antidiagonal and zero else).
\label{thm:symFilteredDG}
\end{theorem}
\begin{proof}

Let $\cfno{\nc:0} := \cf{x-\hh \refx}{\nc:0}$ be
the position-dependent coefficients of the parametrized 
\kernel\ $\flt{x-\hh \refx}$ 
arranged in reverse order. We rewrite Eq.~\eqref{eq:convOnesided} as   
\begin{align}
\label{eq:uConvf}
   \big(u*&\flt{x-\hh \refx}\big)(x) =
   \ub_\Ical \, \ipmat_{\hh \refx} \, 
   \cfno{\nc:0},
   \\
   \ipmat_{\hh \refx} &:= \left[ \int_{\hh (\refx-\kft_{\lst})}^{\hh (\refx-\kft_0)} 
   \, \Nbf{i}{t}{\hh s_{0:\nkdg}} \,
   \Bsp{t}{\hh\refx-\hh \kft_{\nkft-j:\idxlst-j}} \, \drm t 
      \right]_{i\in\Ical,\,j\in\idxft}.
   \notag
\end{align}
Since    
$
\Nsp{hs}{h \cdot} = \Nsp{s}{\cdot}
$ by Eq.~\eqref{eq:DGbasisfuns},
$
\Bsp{hs}{\hh \refx-h\kft_{\nkft-j:\idxlst-j}} 
 =  \frac{1}{h} \Bsp{s}{\refx-\kft_{\nkft-j:\idxlst-j}}
$. Then the change of variable $s := t/h$ yields
\begin{align*}
 \ipmat_{\hh \refx}(i,j) &= 
 \int_{\refx-\kft_{\lst}}^{\refx-\kft_0} \,
 \Nbf{i}{hs}{h\gs_{i:i+\dDG+1}} \, 
 \Bsp{hs}{\hh \refx-h\kft_{\nkft-j:\idxlst-j}} \, h\,\drm s  
 = \ipmat_\refx(i,j). 
\end{align*}
By \thmref{thm:oneSidedCoefs}
\begin{align}
   \cfno{\nc:0}
 &=
\adia 
   \cfno{0:\nc}
\label{eq:fn0xc} \\
 &= 
\adia
\,
M^{-1}_{0,\kft_{0:\lst},\idxft}
\,
   \diag(\left[\begin{matrix}(-1)^\ell \, \binom{\ell+\df+1}{\ell} 
      \end{matrix}\right]_{\ell=0:\nc}) \, 
    \big( (x-\hh \refx)^{0:\nc} \big)^\tr. \notag 
\end{align} 
Eq.~\eqref{eq:uConvf} with Eq.~\eqref{eq:fn0xc} substituted 
equals Eq.~\eqref{eq:symFilteredDG}. 
\end{proof}
The factored representation implies that instead of recomputing the 
filter coefficients afresh for each point $\xx$ of the convolved output as 
in the established \oldMethod,
we simply compute the coefficients corresponding to one \proto\ knot sequence
$\bkft$, scale by $\hh$ as needed and at runtime pre-multiply with the data 
and post-multiply with the vector of shifted monomials as stated in 
Eq.\ \eqref{eq:symFilteredDG}.

Increased multiplicity of an inner knot of the \siac\ kernel reduces its
smoothness, and this, in turn, reduces the smoothness of the filtered output.
By contrast, \thmref{thm:symFilteredDG} shows that
when the \bsiac\ knots are shifted along evaluation points $\xx$ then
\bsiac\ convolution yields a polyonomial, i.e.\
infinite smoothness regardless of the knot multiplicity.
We may view position-dependent filtering as a form of polynomial
approximation. 

Additionally,
the polynomial characterization directly provides a symbolic expression for
the derivatives of the convolved \DG\ output.
\begin{corollary}[Derivatives of \bsiac-filtered \DG\ output]
\begin{equation} \label{eq:derivative}
   \frac{\drm^\ell}{\drm x^\ell}\big(u*\flt{x}\big)(x) = \ub_\Ical 
   \,    
   Q_\refx
   \,
   \diag( h^{-(0:\nc)} )
   \,
   \big( 
   \frac{\drm^\ell}{\drm x^\ell}
   (x-\hh\refx)^{0:\nc} \big)^\tr.
\end{equation}
\end{corollary}
 
The following Corollary shows that for
rational \DG\ break points and rational \kernel\ knots,
we can pre-compute and store the prototype matrix $Q_\refx$
stably in terms of integer fractions.

\begin{corollary}[Rational \bsiac\ convolution coefficients]
With the assumptions and notation of Theorem~\ref{thm:symFilteredDG},
if the basis functions $\Nbf{i}{.}{\gs_{0:\nkdg}}$ 
are piecewise polynomials with rational coefficients,
the shift $\refx$ is rational and  
the sequences $\kft_{0:\nkft}$ and $\gs_{0:\nkdg}$ are rational
then the matrix $Q_\refx$ has rational entries.
\end{corollary}
\begin{proof} 
By \lemref{lem:M0}, the reproduction matrix
$M^{-1}_{0,\bkft,\idxft}$ has rational entries.
Since the integral of a polynomial with rational coefficients over an
interval with rational end points is rational,
Eq.~\eqref{eq:Tcinvariant} implies that 
$\ipmatl{}$ also has rational entries.
This implies that the entries of  $Q_\refx$ are rational.
\end{proof}

\subsection{Application to filtering at boundaries}
\label{sec:filteredDG:onesided}

\newcommand{\kfti}{\widehat{\kft}} 
\newcommand{\kftmax}{\nkft + \df + 1} 
\newcommand{\nkfti}{\nkft_0}
\newcommand{\matTi}[1]{\ipmatl{}^{(#1)}}
\newcommand{\bbfun}[2]{B_{#1}^{(#2)}}
\newcommand{\Tlam}[1]{ \ipmatl{#1} }

We now derive explicit forms of the matrix
$\ipmat_{\refx}$ of \thmref{thm:symFilteredDG} 
when the  \DG\ break points $s_{0:\nkdg}$ are
uniform and the \bsiac\ filters are one-sided. 
For one-sided filters, $\refx$ is replaced by $\refx_L$ and $\refx_R$ 
for the left-sided and right-sided \filter{}s respectively. 
Recalling that convolution reverses the direction of the filter \kernel\
(cf.\ \figref{fig:psiac}),
it is natural to assume that the right-most knot of the left boundary \filter{} 
vanishes when evaluating at the left endpoint $x=a$, i.e.\
$ (h\kft_\lst + x-h\refx_L)|_{x=a}=0$.
Together with the analoguous assumption for right boundary \filter{}s,
this determines 
\begin{equation} \label{eq:onesidedCon}
 \refx_L = \kft_\lst + \frac{a}{h} ,\qquad
 \refx_R =  \kft_0 + \frac{b}{h}.
\end{equation}

\begin{corollary} [\bsiac\ convolution coefficients for uniform
\DG\ knots]
Assume that the \DG\ break point sequence $\gs_{0:\nkdg}$ is uniform, 
 hence after scaling consists of
 consecutive integers.
Without loss of generality, the \DG\ output on each interval $[s_i,s_{i+1}]$ 
is defined in terms of Bernstein-B\'{e}zier polynomials $\bbfun{\ell}{i}$,
$\ell=0:d$, of degree $\dDG$, i.e.\ 
 \begin{equation*}
 \bbfun{\ell}{i}(x) :=
 \begin{cases}
 \binom{\ddg}{\ell} (x-s_i)^\ell (s_{i+1}-x)^{\ddg-\ell} \quad & \text{if $x \in [s_i,s_{i+1}]$}
 \\
 0 \quad &\text{otherwise.}
 \end{cases}
 \end{equation*}

Let $\nkfti$ be the smallest integer greater than or equal 
to $\kft_\lst-\kft_0$.
Then, for $i=0:(\nkfti-1)$, $\ell=0:\ddg$, and $j \in \idxft$
\begin{align}
 \Tlam{L}\big( (\ddg+1)i + \ell, j\big) &=
    \int_{i}^{i+1} \bbfun{\ell}{i}(t) \,  \Bsp{t}{ 
    \kft_{\nkft} - \kft_{\nkft - j:\idxlst-j)}
    } \,  \mathrm{d} t,
    \label{eq:Tleft}
    \\
\Tlam{R}\big( (\ddg+1)i + \ell, j\big) &=
    \int_{\nkfti - i}^{\nkfti - i+1} \bbfun{\ddg-\ell}{\nkfti-i}(t) \,  
    \Bsp{t}{ \kft_{j:j+\df+1}- \kft_0} \,  \mathrm{d} t.
    \notag  
\end{align}
\label{cor:Tlamb:uniform}
\end{corollary}
\begin{proof}
\noindent First we consider $\Tlam{L}$.
In Eq.~\eqref{eq:Tcinvariant}, we change to the variable 
$t = s - \refx_L + \kft_\lst$. 
Abbreviating $\kfti_j :=  \kft_\nkft - \kft_{\nkft-j}$,
since the B-splines are translation invariant,
for 
$0 \leq \kfti_j = \kft_\lst-\kft_{\lst-j} \leq \kft_\lst-\kft_0 \leq \nkfti$,
\begin{equation}
    \Tlam{L}(i,j) = \int_{0}^{\kft_\lst-\kft_0} \, 
    \Nbf{i}{t}{\gs_{0:\nkdg}-\refx_L+\kft_\lst} 
   \, \Bsp{t}{
   \kfti_{j:j+\df+1}
   } \, \drm t.
   \label{eq:Tcinvariant1}   
\end{equation}
%
Since $\gs_0 = \frac{a}{h}$, Eq.~\eqref{eq:onesidedCon} implies that
the first point of the sequence of translated \DG\ break points
$\gs_{0:\nkdg}-\refx_L+\kft_\lst$ equals $0$, i.e.,
$\gs_0 - \refx_L + \kft_\lst = 0$. 
Since the break points are consecutive integers starting from 
$0$ and
$0 \leq \kfti_j \leq \nkfti$
the relevant \DG\ break points are 
$0:\nkft_0$.
%
We re-index the basis functions 
$\Nbf{i}{s}{\gs_{0:\nkdg}-\refx_L+\kft_\lst}$ in terms of 
the Bernstein-B\'{e}zier basis functions $\bbfun{\ell}{i}$.
Since $\bbfun{\ell}{i}$ are non-zero on any interval $[i,i+1]$,
Eq.~\eqref{eq:Tleft} for $\Tlam{L}$ follows from Eq.~\eqref{eq:Tcinvariant1}. 

Now consider $\Tlam{R}$. The change of variable 
$t :=-(s-\refx_R+\kft_0)$ together with the translation
invariance of B-splines imply that
\begin{align}
    \Tlam{R}(i,j) &= \int_{0}^{\kft_\nkft-\kft_0} \, 
    \Nbf{i}{-t}{\gs_{0:\nkdg}-\refx_R+\kft_0} 
   \, \Bsp{-t}{
   \kft_0-\kft_{\nkft-j:\idxlst-j}
   } \, \drm t.
   \notag
   \\
    &= \int_{0}^{\kft_\nkft-\kft_0} \, 
    \Nbf{m-i}{t}{\refx_R-\kft_0-\gs_{\nkdg:0}} 
   \, \Bsp{t}{
   \kft_{j:j+\df+1}-\kft_0
   } \, \drm t.   
   \label{eq:Tcinvariant2}  
\end{align}
Because of Eq.~\eqref{eq:onesidedCon} and $\gs_\nkdg=\frac{b}{h}$,
$\refx_R-\kft_0-\gs_\nkdg=0$. Eq.~\eqref{eq:Tleft} is derived by re-indexing
the basis functions $\phi_{m-i}$ 
as for the left boundary except that we have to count the functions
$\bbfun{\ell}{i}$ backward from the last interval 
$[\kft_{\nkfti-1},\kft_{\nkfti}]$
to the first one $[\kft_{0},\kft_{1}]$ when $i$ and $j$ increase. 
\end{proof}  

The assumption that the \DG\ output is represented in Bernstein-B\'ezier form,
represents no restriction:
the \DG\ output can be represented in any piecewise polynomial form
to arrive at similar formulas.

\section{\bsiac\ \filter s} \label{sec:someFilters}
\newcommand{\nkSR}{\mu}
\newcommand{\lambSR}[1]{\refx_{#1,d}}
\newcommand{\kftv}[1]{{\boldsymbol \kft}_{#1,d}}
\newcommand{\Tstar}{T^*}
\newcommand{\massbb}{\mathcal{M}_{d}}
\newcommand{\Tlr}[1]{\ipmat_{\lambSR{#1}}}
In this section, we first
restate the SRV and the RLKV filters as \bsiac\
filters and compare convolution using the 
\oldMethod\ to using the \newMethod.
Second, to illustrate the generality of the setup,
a new multiple-knot linear filter is defined and 
compared to the SRV and the RLKV filters.

\subsection{An alternative view of the published boundary filters}
We now recast several published boundary filters as special cases of
\bsiac\ \filter s defined 
over shifted knots in the sense of Theorem~\ref{thm:symFilteredDG}. 
The published filters considered in this section 
all chose $\ddg = \df$, i.e.\ the filter has the same degree 
as the \DG\ output. 

\noindent
The boundary SIAC filters RS \cite{siac2003} 
and SRV \cite{siac2011} of reproduction degree $\nc$ 
are both defined over the following shifted knots:
\begin{align} 
&\nkSR := \frac{\nc+d+1}{2}, \quad \lambSR{L}:= a+\nkSR,
\quad \lambSR{R}:= b-\nkSR\notag
\\
 &\kftv{*}(x) := \Big( -\nkSR,-\nkSR+1,\ldots,\nkSR \Big) + x-\lambSR{*}. \label{eq:RSknots}
\end{align}
Here $L$ or $R$ are substituted for $*$, we obtain the left-side
and the right-side filters respectively. 
Both types of filter are associated with a consecutive index sequence $\idxft$. 
However, the two \filter s have different reproduction degrees:
$\nc(\text{RS})=2d$ and $\nc(\text{RV})=4d$. 
The \proto\ knot sequences $\kftv{*}(\lambSR{*})$ are
chosen symmetrically supported about the origin.

In \propref{prop:srvT} below we denote as $\bb_k$ the vectors of 
Bernstein-B\'{e}zier coefficients of the $d+1$ polynomial pieces 
of the uniform B-spline of degree $d$ defined over the knots $0:(d + 1)$.
For example,
\begin{align}
 &d=1: \quad \left[\bb_0 \, \bb_1\right] = 
 \left[ \begin{smallmatrix}
         0 & 1
         \\
         1 & 0
        \end{smallmatrix}
 \right] \notag
 \\
 &d=2: \quad \left[\bb_0 \, \bb_1 \, \bb_2\right] = \frac{1}{2}
 \left[ \begin{smallmatrix}
         0 & 1 & 1 
         \\
         0 & 2 & 0
         \\
         1 & 1 & 0
        \end{smallmatrix}
 \right] 
 \\
 &d=3: \quad \left[\bb_0 \, \bb_1 \, \bb_2 \, \bb_3\right] = \frac{1}{6}
 \left[ \begin{smallmatrix}
         0 & 1 & 4 & 1
         \\
         0 & 2 & 4 & 0
         \\
         0 & 4 & 2 & 0
         \\
         1 & 4 & 1 & 0
        \end{smallmatrix}
 \right].  \notag
\end{align}

\begin{prop} \label{prop:srvT}
Let $\bb_k$ denote the Bernstein-B\'{e}zier coefficients of polynomial pieces 
of a uniform B-spline of degree $d$ defined over the knots $0:(d + 1)$ 
and $\massbb$ the matrix with entries
$\massbb(\ell,j) :=  \frac{1}{2d+1}\binom{d}{\ell} \binom{d}{j} 
\binom{2d}{\ell+j}^{-1}$, $\ell,j=0:d$. Then
  \begin{align}
   \Tlam{L}  := \Tlr{L} 
   = \left[ \begin{smallmatrix}
   \massbb \bb_0 & 0 & \cdots & 0 
   \\
   \massbb \bb_1 & \massbb \bb_0 & \cdots & 0
   \\
   \vdots & \vdots & \cdots & \vdots 
   \\
   \massbb \bb_d & \massbb \bb_{d-1} & \cdots & 0 
   \\
   0 & \massbb \bb_{d}  & \cdots & 0
   \\
   \vdots & \vdots & \cdots & \vdots
   \\
   0 & 0 & \cdots & \massbb \bb_0
   \\
   \vdots & \vdots & \cdots & \vdots
   \\
   0 & 0 & \cdots & \massbb \bb_d
           \end{smallmatrix}        
           \right]
           \label{eq:srvT}
\end{align}
and $\Tlam{R}$ is obtained from $\Tlam{L}$
by reversing the order of the columns and of the rows.
\end{prop} 
\begin{proof}
 Using the notation of  Corollary \ref{cor:Tlamb:uniform}, $n_0 = \nkft = \nc + \ddg$,
 since $\kfti_{0:\nkft}=\kft_\nkft -\kft_{\nkft:0} =  0..(\nc+d)$, 
 \begin{equation}
  \Tlam{L}\big( (d+1)i + \ell, j\big) =
    \int_{i}^{i+1} \bbfun{\ell}{i}(t) \,  \Bsp{t}{ 
    j:j+d+1
    } \,  \mathrm{d} t,
    \label{eq:Tleft1}
 \end{equation}
When $i$ and $\ell$ vary, Eq.~\eqref{eq:Tleft1} gives the $j$th column of 
$\Tlam{L}$. The support of $\Bsp{t}{j:j+d+1}$ contains $d+1$ intervals 
$[j+\rho,j+\rho+1]$, $\rho=0:d$. 
When $[i,i+1] \equiv [j+\rho,j+\rho+1]$, i.e.~when $i=j+\rho$, we can rewrite 
Eq.~\eqref{eq:Tleft1} as 
\begin{align}
   \left[ \Tlam{L}\big( (d+1)i + \ell, j\big) 
 \right]_{\ell=0:d} &=
    \int_{i}^{i+1} \left[\bbfun{\ell}{i}(t) \right]^\tr_{\ell=0:d}
    \,  
    \left[ \bbfun{j}{i}(t) \right]_{j=0:d} \bb_\rho
    \,  \mathrm{d} t,
    \\
    &=
    \left[ \int_0^1 \bbfun{\ell}{0}(t) \, \bbfun{j}{0}(t) \, \drm t 
\right]_{\ell=0:d,j=0:d} \bb_\rho  
    =
    \massbb \bb_\rho,
    \notag
\end{align}
where 
\begin{align}
   \massbb(\ell,j) = \binom{d}{\ell} \, \binom{d}{j} 
   \, \binom{2d}{\ell+j}^{-1} \,
 \int_0^1 \binom{2d}{\ell+j} x^{\ell+j} (1-x)^{2d-\ell-j}\, \drm t.
 \label{eq:Mdkl}
\end{align}
Since the integral in Eq.~\eqref{eq:Mdkl} equals $\frac{1}{2d+1}$, we have 
derived the formula for $\massbb(\ell,j)$ in Eq.~\eqref{eq:srvT}.
The formula for $\Tlam{L}$ follows,
because the entries of the $j$th column vanish when the two intervals
$[i,i+1]$ and $[j+\rho,j+\rho+1]$ do not overlap. 

The formula for $\Tlam{R}$ in Eq.~\eqref{eq:Tleft} is that of 
$\Tlam{L}$ except that the B-splines appear in reverse order.
Therefore, reversing the column order of $\Tlam{L}$ 
and then the row order of the result yields $\Tlam{R}$.
 \end{proof}

\noindent
The index sequence $\idxft$ of the boundary \filter\ RLKV is non-consecutive. 
The left and right \filter s are of degree $2\ddg + 1$ 
and are defined over the following shifted knots
\begin{align} 
 \kftv{L}(x) &:= \Big( -\nkSR,\ldots,\nkSR-1,
 \underbrace{\nkSR,\ldots,\nkSR}_{\text{$\ddg+1$ times}} \Big) + x-\lambSR{L},
 \label{eq:RLKVknotsL}
 \\
 \idxft_L &:= \{1:(2\ddg+1),3\ddg+1\};
 \notag
 \\
\kftv{R}(x) &:= \Big(
 \underbrace{-\nkSR,\ldots,-\nkSR}_{\text{$\ddg+1$ times}},
  -\nkSR+1,\ldots,\nkSR,
 \Big) + x-\lambSR{R}, \label{eq:RLKVknotsR} 
  \\
 \idxft_R &:= \{1,\ddg:(3\ddg+1)\}.
 \notag
\end{align}
The prototype knot sequence  
$\kftv{L}(\lambSR{L})$ is chosen to have symmetric 
 support about the origin. 
 \begin{prop} \label{prop:rlkvT}
Let $\massbb$ be defined as in Proposition~\ref{prop:srvT} 
and let $\massbb^{(1)}$ denote its first column.
Then
  \begin{align}
   \Tlr{L} = \left[ \begin{smallmatrix}
   (d+1)\massbb^{(1)} & \massbb \bb_0 & 0 & \cdots & 0 
   \\
   0 & \massbb \bb_1 & \massbb \bb_0 & \cdots & 0
   \\
   \vdots & \vdots & \vdots & \cdots & \vdots 
   \\
   0 & \massbb \bb_d & \massbb \bb_{d-1} & \cdots & 0 
   \\
   0 & 0 & \massbb \bb_{d}  & \cdots & 0
   \\
   \vdots & \vdots & \vdots & \cdots & \vdots
   \\
   0 & 0 & 0 & \cdots & \massbb \bb_0
   \\
   \vdots & \vdots & \vdots & \cdots & \vdots
   \\
   0 & 0 & 0 & \cdots & \massbb \bb_d  
           \end{smallmatrix}        
           \right]
           \label{eq:rlkvT}     
  \end{align}
$\Tlam{R}$ is obtained from $\Tlam{L}$
by reversing the order of the columns and of the rows.
 \end{prop}
 \begin{proof}
The proof is the same as that of Proposition~\ref{prop:srvT}
except that the additional B-spline of the \kernel\ contributes
the first column of \eqref{eq:rlkvT}. 
 \end{proof}

\subsection{A multiple-knot boundary \bsiac\  \kernel}
\label{sec:leastdeg}
Since their introduction in the seminal paper \cite{siac2003},
boundary \filter s have been given the same  degree as the symmetric kernel 
\cite{siac2014} -- 
perhaps to guarantee the same smoothness near the boundary 
as in domain interior.
Indeed, the authors of \cite{siac2011} numerically observed and predicted that 
SRV-filtered \DG\ outputs would be as smooth as the SRV \filter.
Theorem~\ref{thm:symFilteredDG} implies not only that a \bsiac\ \filter\ 
need not have the same degree as the symmetric \filter, but
confirms smoothness of the output beyond the conjecture:
it shows that \bsiac\ \filter s may have multiple knots without reducing the
the infinity smoothness of the filtered \DG\ output.

\newcommand{\knotsize}{\nkSR} 
We define a new \emph{multiple-knot left-sided \filter} $\sk_L$ and 
a right-sided \filter\ $\sk_R$ each of degree $\df=1$ (hence 
double-knot filters)  respectively over the knot sequences
\begin{align} 
   \knotsftv_L &:= x-\refx_L + \big(-\knotsize,\ldots,\knotsize-3,
   \knotsize-2,\knotsize-1,\knotsize-1,\knotsize,\knotsize \big), \
   \knotsize := \frac{3\ddg+1}{2},
   \notag
   \\
   \knotsftv_R &:= x-\refx_R + \big(-\knotsize,-\knotsize,-\knotsize+1,-\knotsize+1,
   -\knotsize+2,-\knotsize+3\ldots,\knotsize\big).
   \label{eq:LRknots}
\end{align}
Since the degree of the \filter s is $1$, their \orderft{} is $3d+2$
even though they have the same support size as the symmetric \siac\ \filter\
that one might apply in the interior of the $\DG$ domain.
The two double-knots 
increase the reproduction degree but not the support.

\begin{figure}[ht!]
\def\figw{.45\textwidth}
 \centering
 \begin{tabular}{cc}
  \includegraphics[width=\figw]{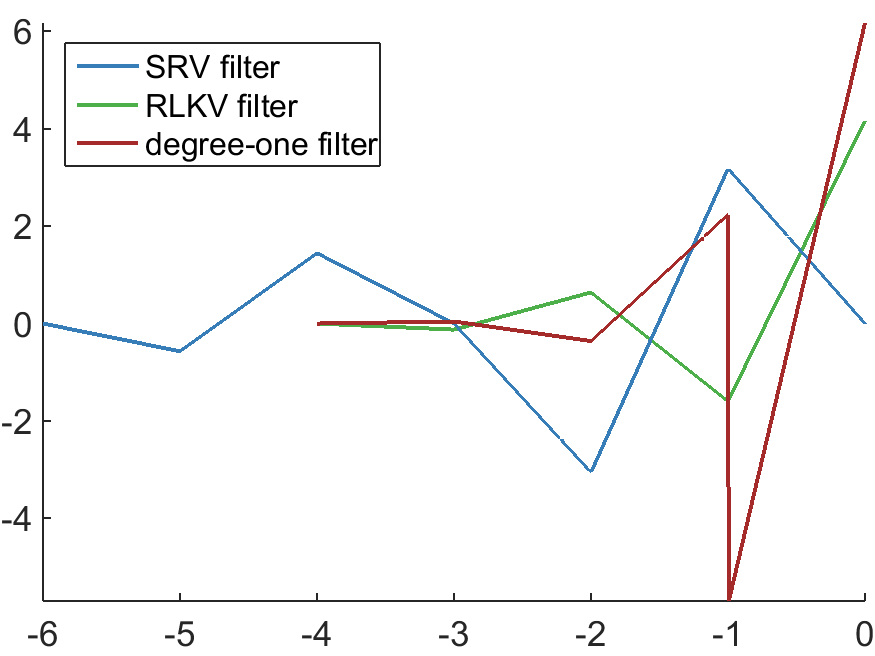}
  &
  \includegraphics[width=\figw]{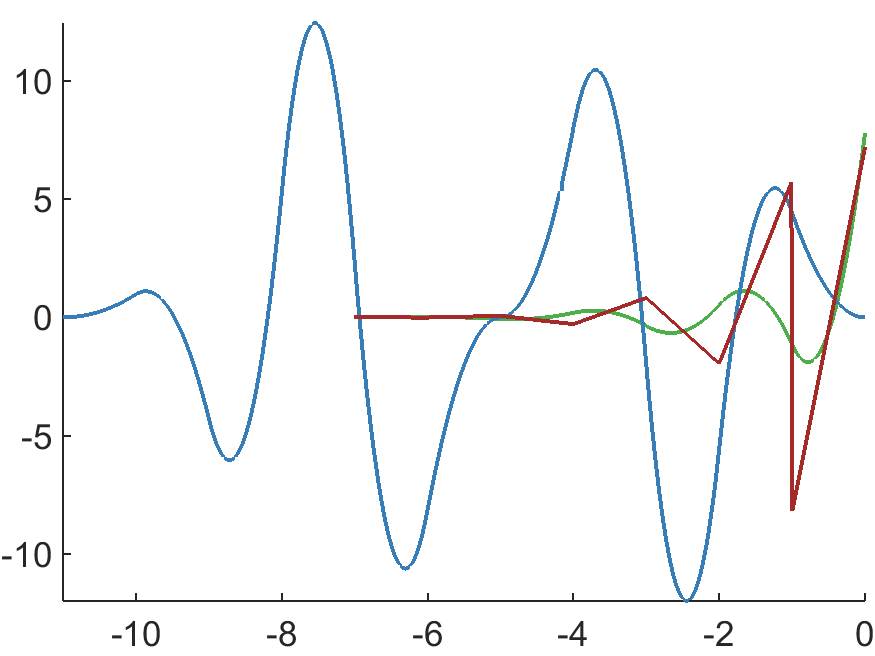}
  \\
  (a) \filter s for $\ddg=1$
  &
  (b) \filter s for $\ddg=2$
  \\
  \includegraphics[width=\figw]{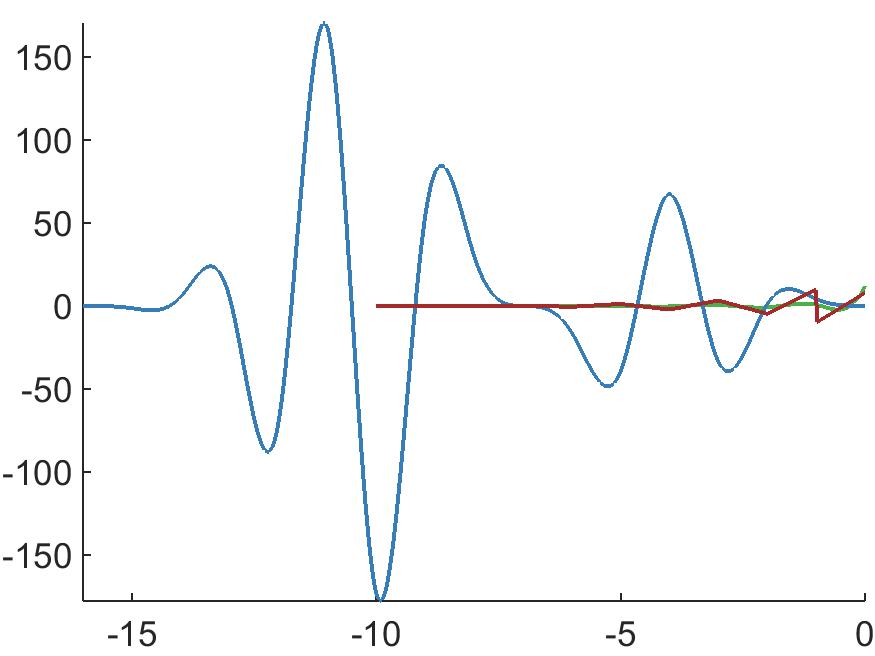}
  &
  \includegraphics[width=\figw]{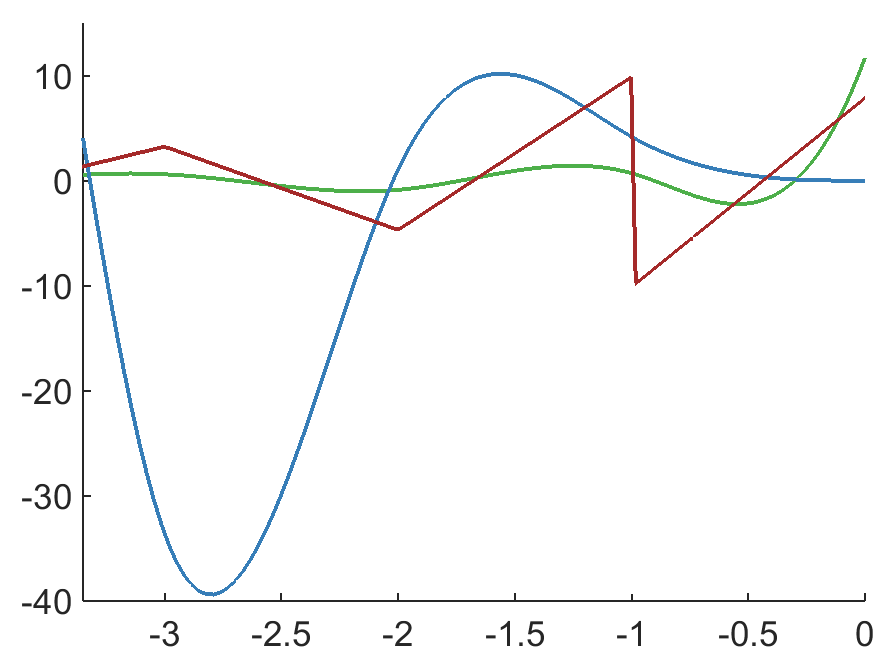}  
  \\
  (c) \filter s for $\ddg=3$
  &
  (c') right zoom of (c)
 \end{tabular}
\caption{Graphs of the three \filter s defined at the left boundary $x=a$.
Note that the degree of \cite{siac2011}  and \cite{siac2014} increases with
$\ddg$ while the degree of the multiple-knot kernel stays linear.}
\label{fig:filters}
\end{figure}  

\subsection{Examples} 
\label{sec:examples}  
\figref{fig:filters} graphs instances of the \cite{siac2011},
\cite{siac2014} 
and of the new multiple-knot \filter\ of degree-one.
With the help of the following two examples
we numerically verify the symbolic expression \eqref{eq:symFilteredDG}.
We demonstrate improved stability of the \newMethod{} over the \oldMethod\
and we illustrate a possible use of a \filter\ of 
degree $\df \ne \ddg$ (we will choose $\df=1<\ddg$) and 
illustrate a possible use of inner knots with higher multiplicity. 
Both examples are special cases of the canonical Eq.~\eqref{eq:hypEqs}
%
with
\begin{align}
   \kappa(x,\tm) \equiv 1,\quad
   \rho(x,\tm) \equiv 0,\quad
   0 \leq \tm \leq \tmend.
   \label{eq:kr}
\end{align}

\noindent
{\bf Example 1} Consider Eq.~\eqref{eq:hypEqs} with specialization
\eqref{eq:kr}, Dirichlet boundary conditions and  $\tmend:=\frac{1}{16}$.
The exact solution is  
$\dgout(x,\tmend) = \frac{7}{10}\sin(\pi \sqrt{\frac{10}{7}}(x-\tmend))$.

\noindent 
{\bf Example 2} Consider Eq.~\eqref{eq:hypEqs} with specialization
\eqref{eq:kr}, periodic boundary conditions and  $\tmend:=1$,
i.e.\ after a sequence of time steps.
The exact solution is $\dgout(x,\tmend) = \sin(2\pi(x-\tmend))$.
\begin{figure}[h!t] \footnotesize
\def\figw{.32\linewidth}
\begin{tabular}[c]{ccc}
  \subfigure[ Example 1 \DG\ $\ddg=2$ output error] { 
 \includegraphics[width=\figw]{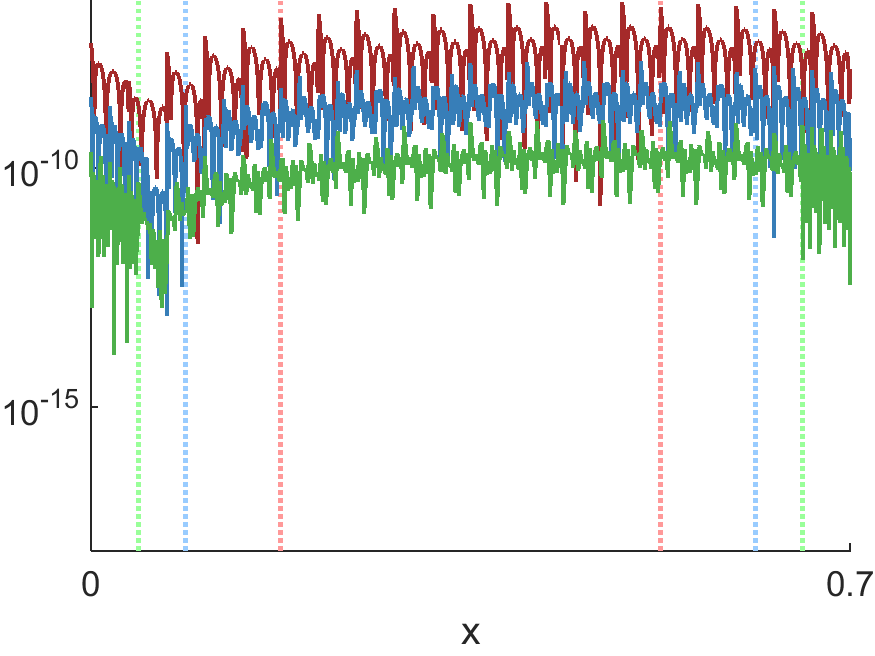}}
 &
  \subfigure[ \cite{siac2011} numerical]{
  \label{fig:oldM1}
  \includegraphics[width=\figw]{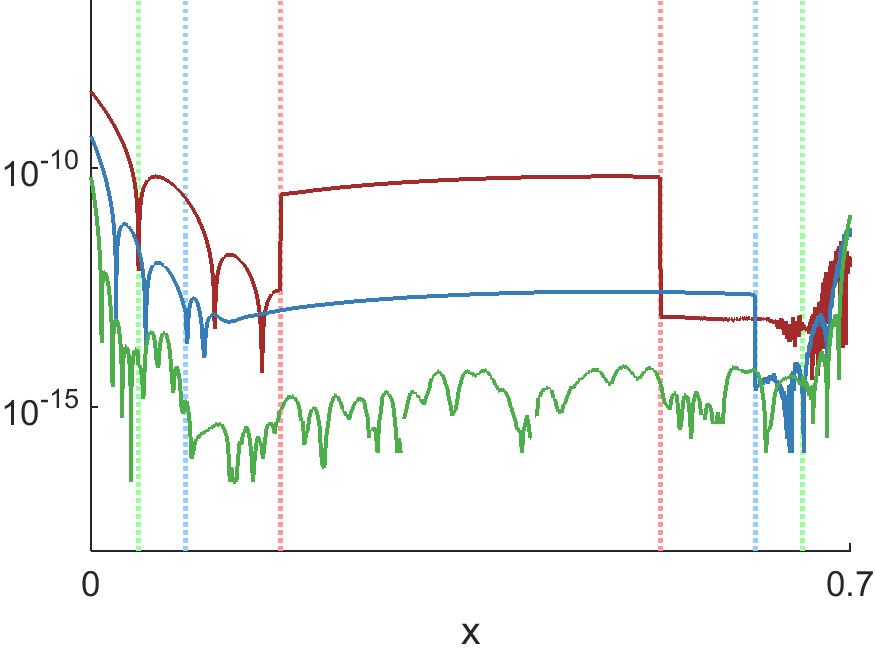}  }
 &
  \subfigure[ \cite{siac2011} symbolic]{
  \label{fig:newM1}
  \includegraphics[width=\figw]{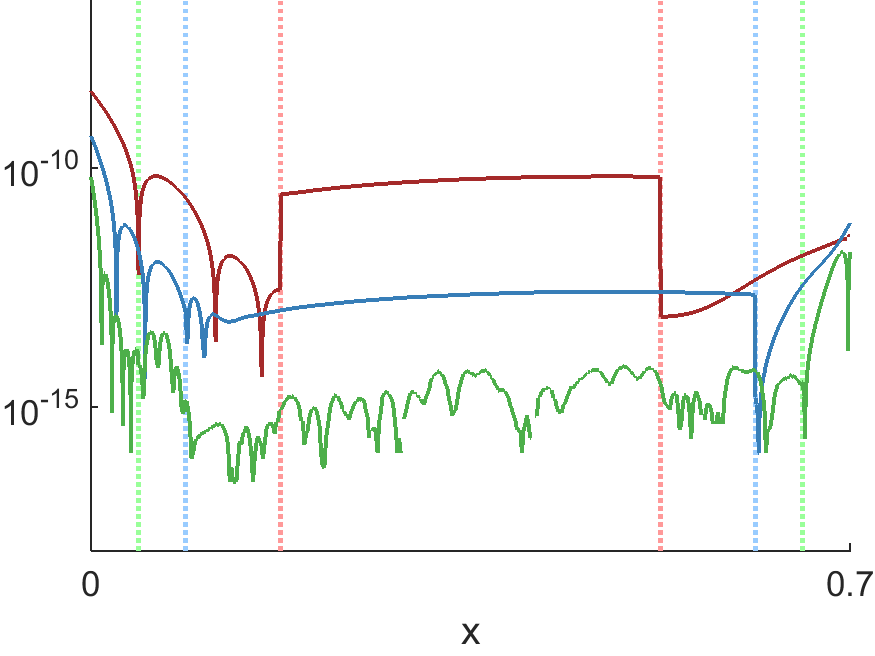}  }
  \\
 &
  \subfigure[ Left zoom of \ref{fig:oldM1}]{
  \includegraphics[width=\figw]{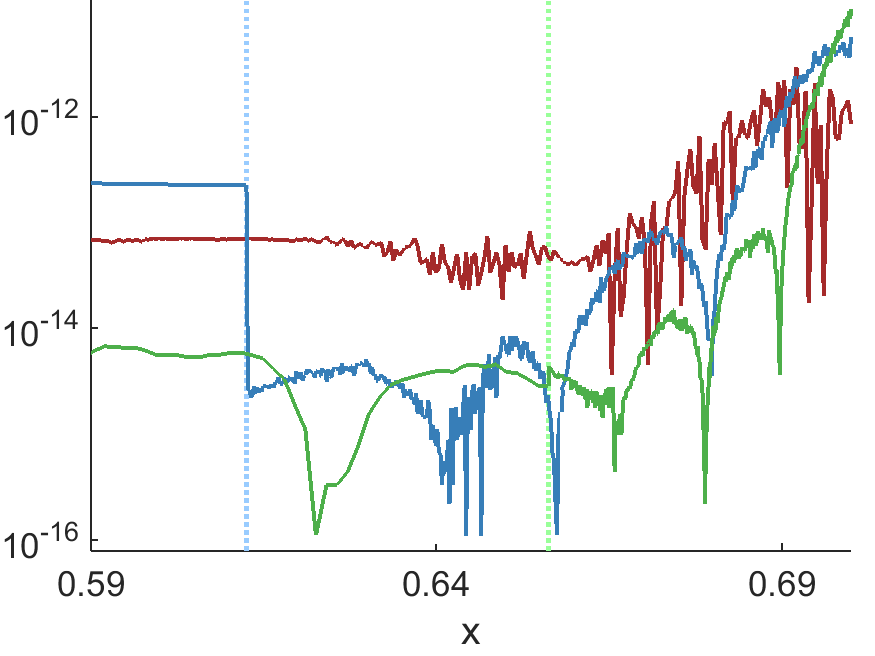}}
 &
  \subfigure[ Left zoom of \ref{fig:newM1}]{
  \includegraphics[width=\figw]{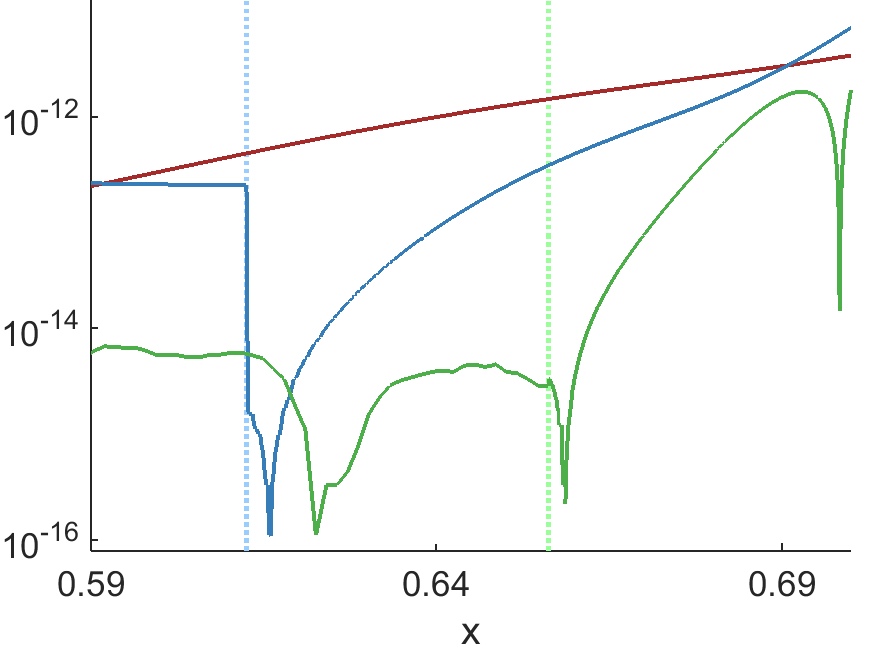}   }
  \\
  &
  \subfigure[\cite{siac2014}]{
 \includegraphics[width=\figw]{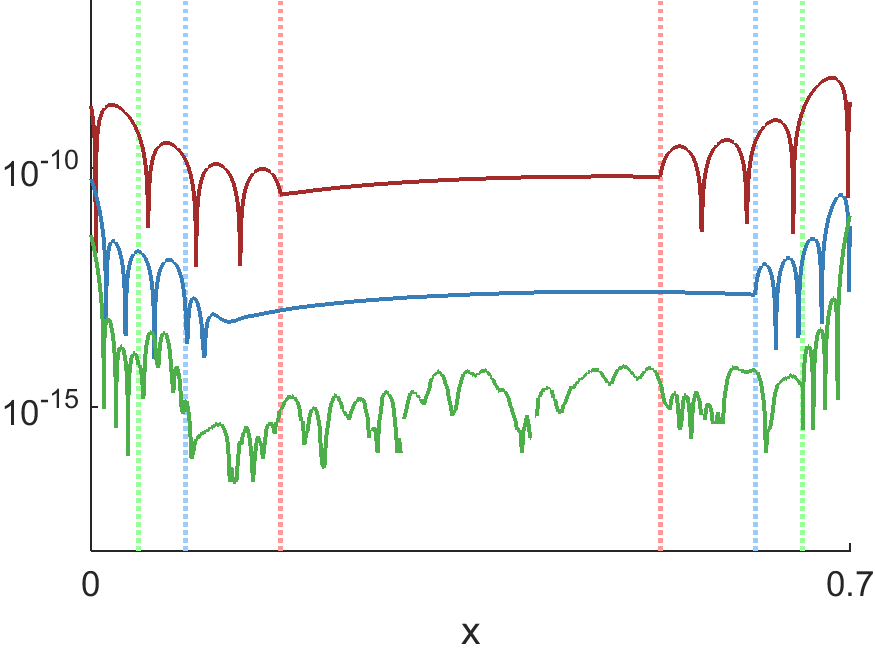}  }
  &
  \subfigure[new multiple-knot filter]{
  \includegraphics[width=\figw]{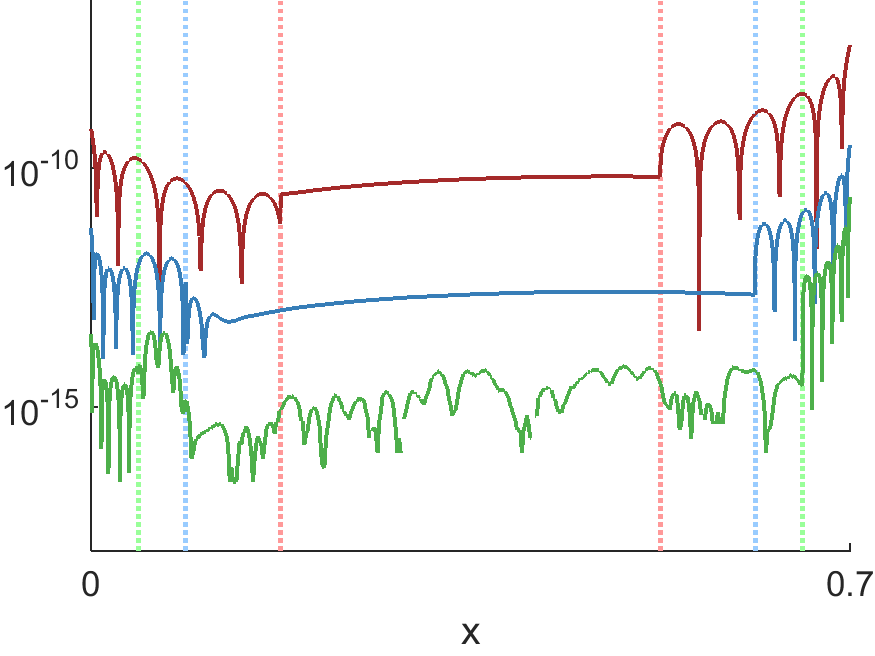}  }
\end{tabular}
\caption{Point-wise errors of \DG\ approximations of Example 1
when $\ddg=2$.
The three graphs in each subfigure correspond, from top to bottom,
(red, blue, green) to $\nkdg = 20, 40$ and $80$ \DG\ break points.
}
\label{fig:dDG2T1}
\end{figure}

\begin{figure}[h!t] \footnotesize
\def\figw{.32\linewidth}
\begin{tabular}[c]{ccc}
  \subfigure[Example 2 \DG\ $\ddg=3$ output error] { 
 \includegraphics[width=\figw]{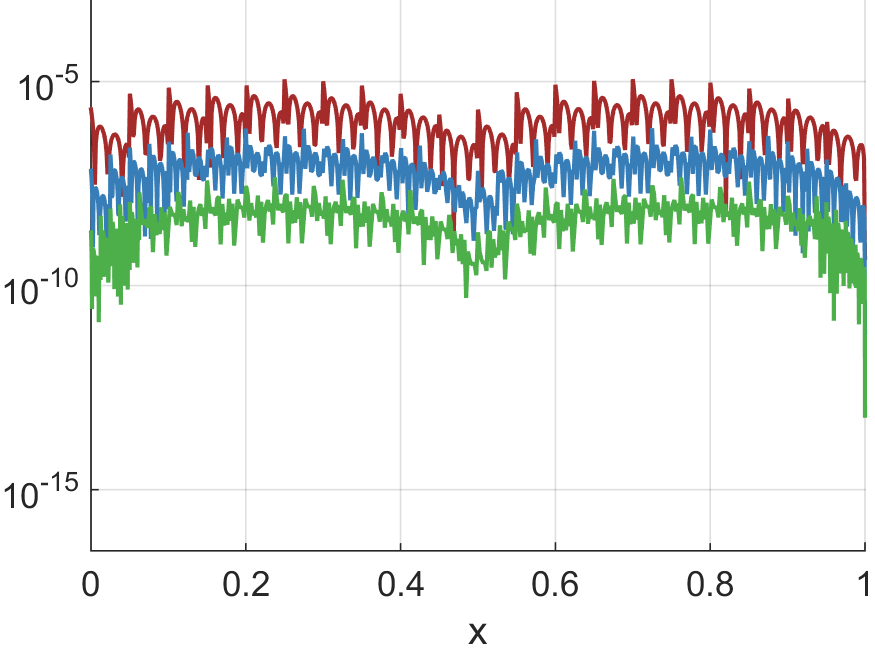}}
 &
  \subfigure[ \cite{siac2011} numerical]{
  \label{fig:oldM}
  \includegraphics[width=\figw]{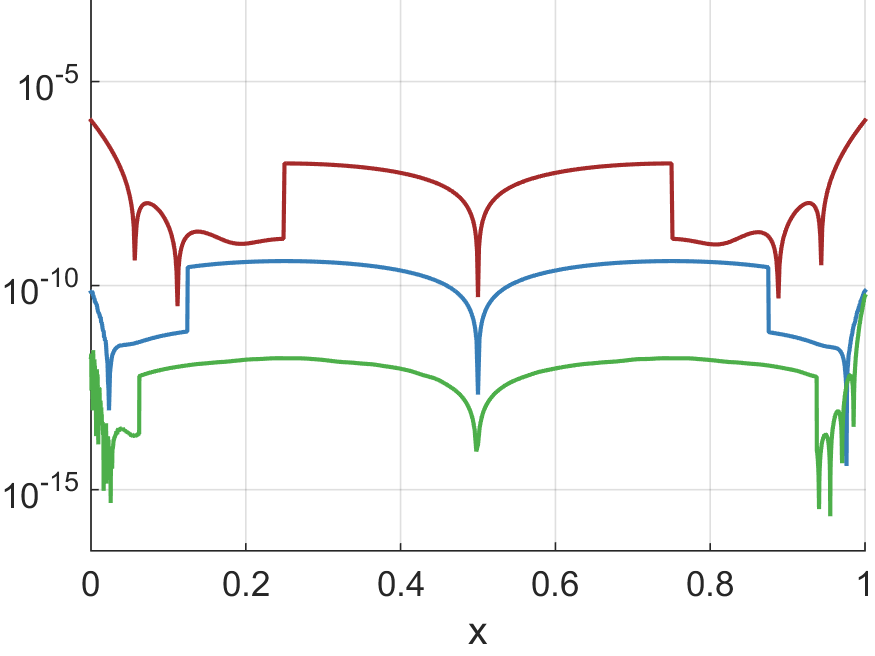}}
 &
  \subfigure[ \cite{siac2011} symbolic]{
  \label{fig:newM}
  \includegraphics[width=\figw]{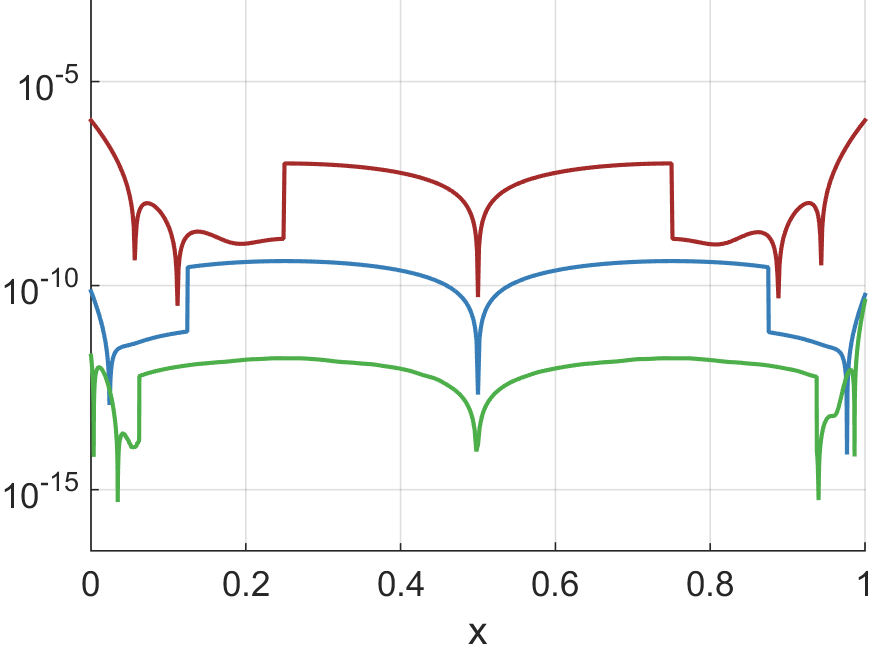}}
  \\
 &
  \subfigure[ Left zoom of \ref{fig:oldM}]{
  \includegraphics[width=\figw]{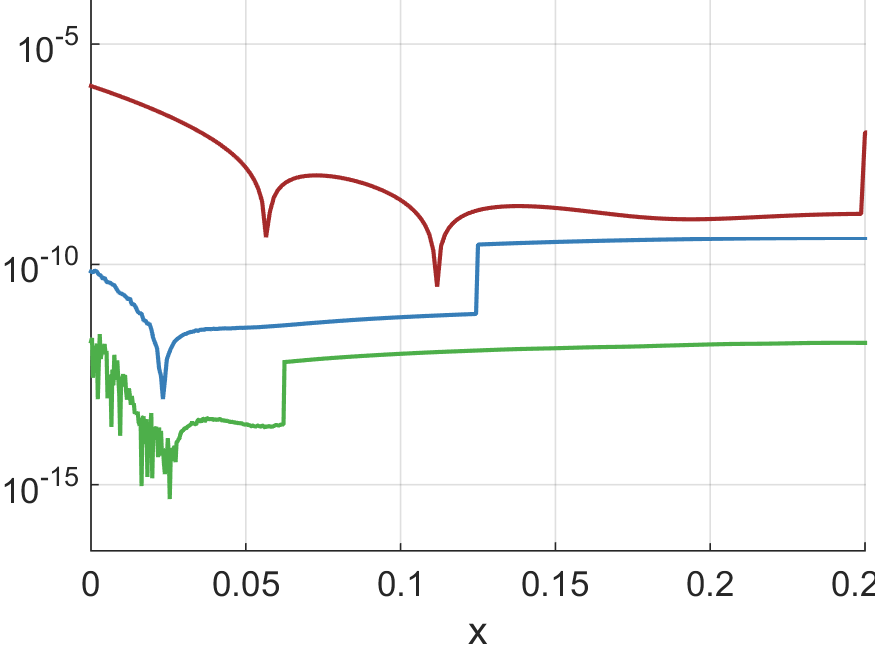}}
 &
  \subfigure[ Left zoom of \ref{fig:newM}]{
  \includegraphics[width=\figw]{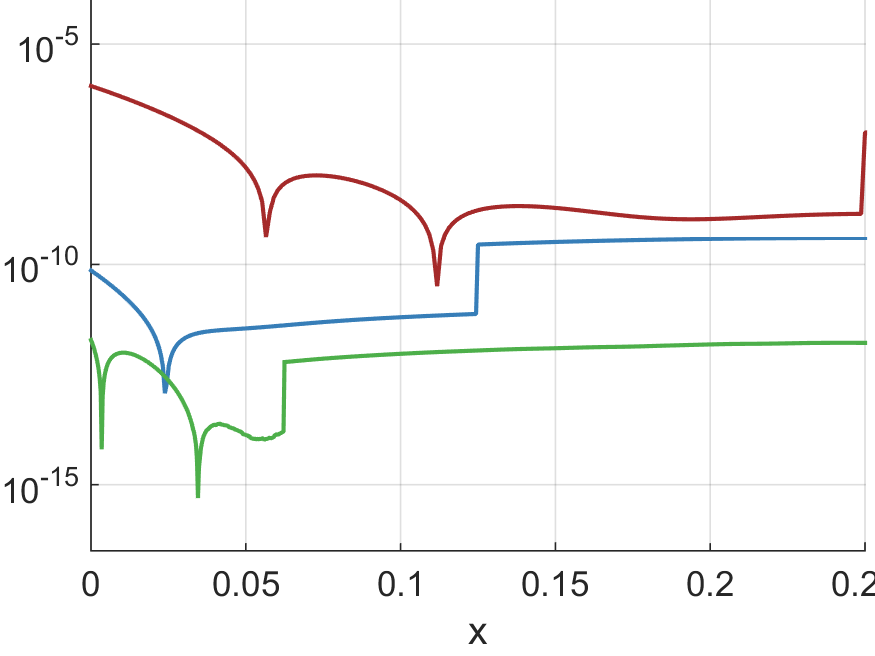}   }
  \\
  &
  \subfigure[\cite{siac2014}]{
  \includegraphics[width=\figw]{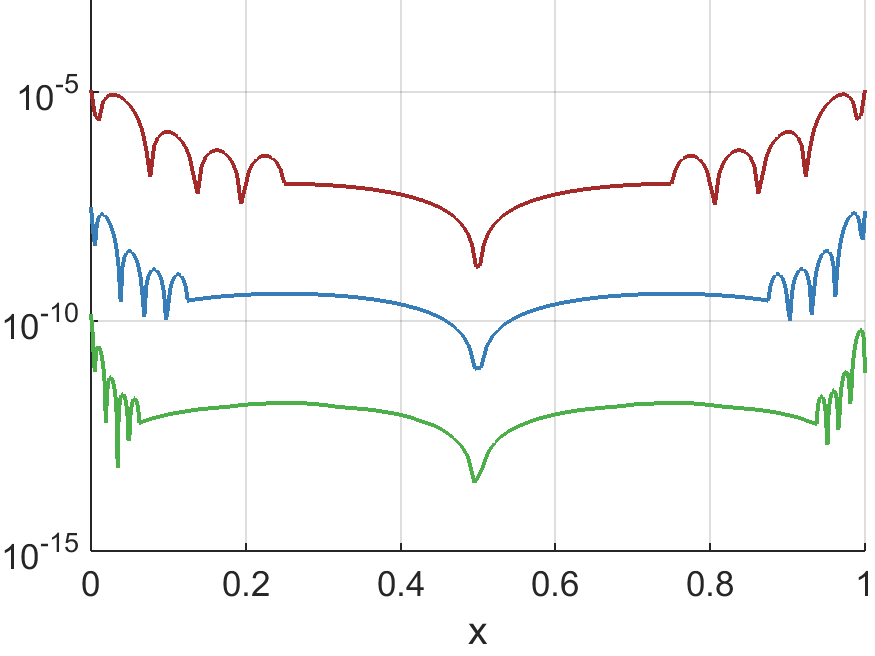} }
  &
  \subfigure[new multiple-knot filter]{
  \includegraphics[width=\figw]{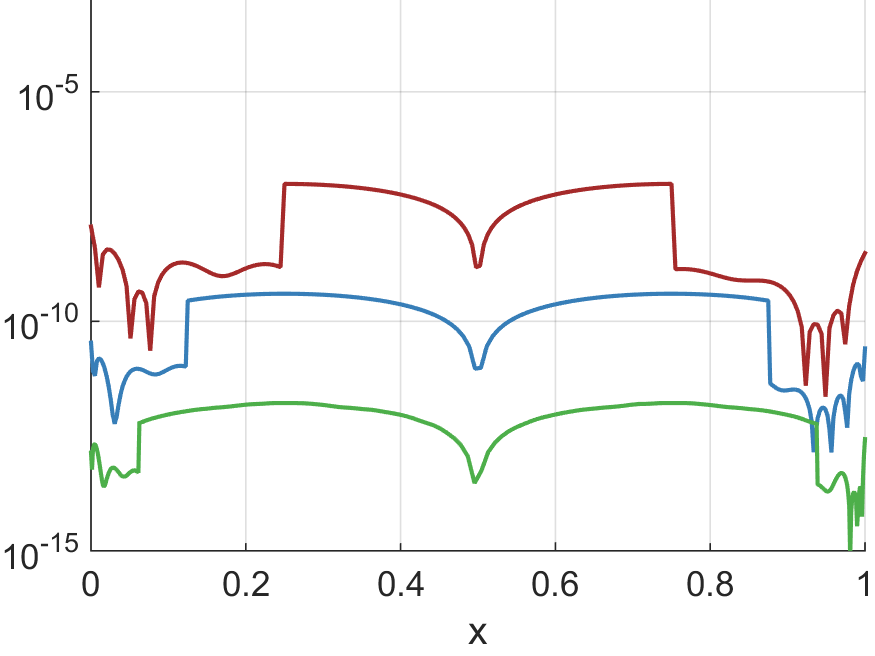}  }
\end{tabular}
\caption{Point-wise errors of \DG\ approximations of Example 2
when $\ddg=3$.
The three graphs in each subfigure correspond, from top to bottom,
(red, blue, green) to $\nkdg = 20, 40$ and $80$ \DG\ break points.
}
\label{fig:dDG3T1}
\end{figure}

We used upwind numerical flux \cite{DGbook} and solved the
resulting ordinary differential equations with the help of
a standard fourth-order
four stage explicit Runge-Kutta method (ERK) \cite[Section 3.4]{DGbook}.
Figures~\ref{fig:dDG2T1} and \ref{fig:dDG3T1} show the
error of the \DG\ approximations of Example 1 and Example 2, respectively,
when using different post-filters.
 
For \cite{siac2011}, the \oldMethod{} introduces high oscillations 
in the point-wise errors near the boundaries for both example problems.
The increase of the oscillations with decreasing mesh size 
is the result of near-singular matrices:
when $\lstDG=80$ and $\ddg=3$, the matrices $M$ have condition numbers
near the limit where MATLAB declares them singular.
However, the \newMethod{} yields exact results. 

For the RLKV \filter s we only show the result of the \newMethod.
The difference between the \oldMethod{} and the \newMethod\ is less than $10^{-13}$
confirming on one hand the stability of the RLKV filter and 
on the other hand the correctness of 
reformulation according to \thmref{thm:symFilteredDG}.
RLKV is juxtaposed with the our
degree-one multiple-knot \filter. For Example 1, displayed in
\figref{fig:dDG2T1}, the multiple-knot filter
has a slightly higher right-sided error (and a smaller left-sided one)
whereas for Example 2 \figref{fig:dDG3T1}, the multiple-knot \filter\ 
has a clearly lower error.

\section{Conclusion}
The state-of-the-art approach for computing a filtered DG output 
\cite{siac2003,siac2011,siac2012,siac2014} 
consists of,
at each evaluation point $x\in[a,b]$,
(i) assembling then inverting the reproduction matrix 
to obtain the coefficients of the position-dependent boundary kernel and
then
(ii) calculating the convolution integral by Gaussian quadrature 
to obtain the filtered \DG\ output.
Compared to that approach, the \bsiac\ filters according to
\thmref{thm:symFilteredDG} provide for a more stable, flexible, versatile 
and efficient approach:
\begin{itemize}
\item[\itmsym] {\em Stability}:
The knots $\kft_{0:\nkft}$ can be chosen freely,
e.g.\ so that the reproduction matrix $M$ is sufficiently regular. 
When the filter knots are rational, the entries of the inverse of $M$ can be 
pre-computed exactly as fractions of integers.
This avoids the need for
repeatedly inverting near-singular constraint matrices at run-time.
\item[\itmsym] {\em Flexibility}: 
The \bsiac-filtered \DG\ output is a single polynomial piece for the 
length of application of the \bsiac\ filter.
The \bsiac\ kernel can be of any degree 
and it can be defined over a sequence of knots of any multiplicity.
\item[\itmsym] {\it Versatility}: The polynomial characterization of the 
\bsiac-filtered \DG\ output directly yields, for example,
an explicit expression of its derivatives.
\item[\itmsym] {\em Efficiency}: 
Given the \proto\ knot sequence and subset of B-splines on that knot sequence
chosen to define a filter, the convolution matrix
$Q_\refx$ can be pre-computed once and for all. 
Multiplication with a simple diagonal matrix, yields the matrix
for a scaled filter knot sequence. \\
Given the vector $\ub_\Ical$ of coefficients of the \DG\ output,
the vector of polynomial coefficients $\ub_\Ical \, Q_\refx$
of the filtered output (a vector of size 
$\nc+1$) can be computed per data set, for all further convolution operations.
\\
The convolution for a point $\xx$ near the boundary 
then simplifies to a dot product of two vectors of size $\nc+1$. 
\end{itemize}
\bibliographystyle{alpha}  
\bibliography{p}
\end{document}